\documentclass[11pt,twoside]{amsart}
\usepackage{amsmath, amsthm, amscd, amssymb, color, bbm}
\usepackage[bookmarksnumbered, plainpages, pagebackref, colorlinks]{hyperref}
\usepackage{tikz-cd}

\numberwithin{equation}{section}

\theoremstyle{definition}
\newtheorem{ntn}{Notation}[section]

\theoremstyle{plain}
\newtheorem{lem}[ntn]{Lemma}
\newtheorem{prp}[ntn]{Proposition}
\newtheorem{thm}[ntn]{Theorem}
\newtheorem{cor}[ntn]{Corollary}
\newtheorem{conj}[ntn]{Conjecture}
\theoremstyle{definition}

\newtheorem{rem}[ntn]{Remark}

\newcommand{\z}{\mathbb{Z}}

\newcommand{\q}{\mathbb{Q}}

\newcommand{\R}{\mathbb{R}}
\newcommand{\kk}{\mathbbm{k}}

\newcommand{\K}{\mathcal{K}}

\newcommand{\EE}{\mathcal{E}}
\newcommand{\DD}{\mathcal{D}}
\newcommand{\B}{\mathcal{B}}

\newcommand{\lan}{\langle}
\newcommand{\ran}{\rangle}
\newcommand{\GL}{\mathrm{GL}}
\newcommand{\SL}{\mathrm{SL}}
\newcommand{\GM}{\mathrm{GM}}

\newcommand{\inc}{{\rm inc}}
\newcommand{\id}{{\rm id}}
\newcommand{\tors}{{{\rm Tor}_1^{\z}}}
\newcommand{\zzz}{\z[\frac{1}{2}]}

\newcommand{\half}{{\Big[\frac{1}{2}\Big]}}
\newcommand{\sixth}{{\Big[\frac{1}{6}\Big]}}

\newcommand{\arr}{\rightarrow}
\newcommand{\larr}{\longrightarrow}
\newcommand{\harr}{\hookrightarrow}
\newcommand{\se}{\subseteq}

\newcommand{\mt}{\mapsto}
\newcommand{\two}{\twoheadrightarrow}

\newcommand{\fn}{F^n}
\newcommand{\ffn}{F^{n-2}}
\newcommand{\fff}{{F^\times}}
\newcommand{\eee}{E}

\newcommand{\stabe}{{\rm Stab}}
\newcommand{\diag}{{\rm diag}}
\renewcommand{\ker}{{\rm ker}}
\newcommand{\coker}{{\rm coker}}
\newcommand{\im}{{\rm im}}
\newcommand{\ind}{{\rm ind}}

\renewcommand{\char}{{\rm char}}

\newcommand{\zn}{\z\Big[\frac{1}{(n-2)!}\Big]}
\newcommand{\znn}{\z\Big[\frac{1}{(n-1)!}\Big]}

\newcommand{\nn}{\Big[\frac{1}{(n-2)!}\Big]}

\newcommand{\zmm}{\z\Big[\frac{1}{(m-1)!}\Big]}

\newcommand {\mtx}[4]
{\left(\!\!\!
\begin{array}{cc}
#1 & #2   \\
#3 & #4
\end{array}
\!\!\!\right)}

\newtheoremstyle{athm}
  {}
  {}
  {\itshape}
  {}
  {\scshape}
  {}
  {.5em}
  {\thmnote{#3}}
\theoremstyle{athm}
\newtheorem*{athm}{}

\begin{document}

\title[Homology of general linear groups]{Homology of $\GL_n$ over infinite fields outside the stability range}

\author{Behrooz Mirzaii}

\thanks{{\scriptsize
\hskip -0.4 true cm MSC(2020): Primary: 19D55, 19D45; Secondary: 20J06.
}}

\begin{abstract}
For an infinite field $F$, we study the kernel of the map
\[
\begin{array}{c}
H_{n}(\GL_{n-1}(F),\zn) \arr H_{n}(\GL_{n}(F),\zn),
\end{array}
\]
and the  cokernel of
\[
\begin{array}{c}
H_{n+1}\Big(\GL_{n-1}(F),\zn\Big) \arr H_{n+1}\Big(\GL_{n}(F),\zn\Big).
\end{array}
\]
We give conjectural 
estimates of  these kernels and cokernels  and prove our conjectures for $n\leq 4$.


\end{abstract}

\maketitle
\section*{ Introduction}
\renewcommand{\thepage}{\arabic{page}}
\setcounter{page}{1}

Let $F$ be a field. For any positive integer $r$,  $\GL_r(F)$ embeds naturally in $\GL_{r+1}(F)$. The sequence of group 
embeddings $\GL_1(F) \se \GL_2(F) \se \GL_3(F) \se \cdots$ induces the sequence of homomorphisms of homology groups
\[
H_n(\GL_1(F), \z) \arr H_n(\GL_2(F), \z) \arr H_n(\GL_3(F), \z) \arr \cdots.
\]
These homology group appear in many areas of Algebra and Geometry. Unfortunately it is hard to calculate them explicitly.
Therefore all results allowing to compare them for different values of $n$ become quite important.

By an unpublished work of Quillen, if $F$ has more than two elements, then
\begin{equation}\label{st}
H_n(\GL_r(F), \z) \arr H_n(\GL_{r+1}(F), \z)
\end{equation}
is surjective for $r\geq n$ and is bijective for $r\geq n+1$. (Quillen's proof appears in his unpublished note \cite[pp.~1--15]{quillen-1974}. 
Unfortunately the first and the second pages of this note are unreadable, which makes it hard to follow the proof. But nevertheless,
see \cite[Theorem A]{S-W2020} for an exposition of Quillen's argument. Quillen's result also follows from \cite[Theorem~A]{G-K-R2018}.)

With a different method, Suslin showed that 
if the field is infinite, then (\ref{st}) is an isomorphism for $r\geq n$. Moreover, he showed that the cokernel of  
\[
H_n(\GL_{n-1}(F),\z)\arr H_n(\GL_{n}(F),\z)
\] 
is isomorphic to $K_n^M(F)$, the $n$-th Milnor $K$-group  of $F$ \cite[Theorem 3.4]{suslin1985}
(see also Theorem \ref{suslin} below). 

Not much is known about  the map (\ref{st}) outside the stability range, i.e. when $r< n$. The following conjecture is attributed to Suslin
(see \cite[Problem~4.13]{sah1989}, \cite[Remark~7.7]{B-Y1994}, \cite[Conjecture~2]{dejeu2002}).

\begin{athm}[{\bf Suslin's Injectivity Conjecture.}]
For any infinite field $F$ and  any $n>r$, the natural map $H_n(\GL_r(F), \q) \arr H_n(\GL_{r+1}(F), \q)$
is injective.
\end{athm}

The conjecture is trivial for $n=1, 2$. It was known for number fields by the work of Borel and Yang  \cite[Corollary 7.6]{B-Y1994}.
It was proved for $(n=3, r=2)$ by Sah \cite[Remark 3.19]{sah1989} and Elbaz-Vincent  \cite[Theorem 1.22]{elbaz1998} and for 
$(n=4, r=3)$ by the author \cite[Theorem 3]{mirzaii2008}. 

Recently, the conjecture has been proved for $r=n-1$ by Galatius, Kupers and 
Randal-Williams in \cite[Theorem~9.18]{G-K-R2020} (see Theorem \ref{G-K-R} below for a precise statement.)

Beside this, Galatius, Kupers and Randal-Williams have studied the cokernel of the map
\[
H_{n+1}(\GL_{n-1}(F),\q)\arr H_{n+1}(\GL_{n}(F),\q).
\]
 They showed that the group 
\[
\bigoplus _{n\geq 1} H_{n+1}(\GL_n(F), \GL_{n-1}(F),\q)
\]
has a natural $K_\ast^M(F)_\q$-module structure and explained how to generate this module efficiently 
\cite[Theorem~D or Theorem~9.5]{G-K-R2020}. This suggests that the above map has somewhat complicated
cokernel in general.

In this article we study the kernel of
\[
\begin{array}{c}
H_{n}(\GL_{n-1}(F),\zn)\arr H_{n}(\GL_{n}(F),\zn),
\end{array}
\] 
and the cokernel of
\[
\begin{array}{c}
H_{n+1}(\GL_{n-1}(F),\zn)\arr H_{n+1}(\GL_{n}(F),\zn).
\end{array}
\] 
We show that both are related to the second homology of a complex, which we introduce now. The chain of maps
\begin{equation}\label{seq1}
\fff^{\otimes n}\! \otimes_\z K_0^M(F)\!\overset{\delta_n^{(n)}}{\larr}\!
\fff^{\otimes (n-1)}\! \otimes_\z K_1^M(F)\!\overset{\delta_{n-1}^{(n)}}{\larr}\!
\!\cdots\!  \overset{\delta_{3}^{(n)}}{\larr} \!\fff^{\otimes 2}\! \otimes_\z K_{n-2}^M(F)
\end{equation}
\[
\overset{\delta_{2}^{(n)}}{\larr} \fff \otimes_\z K_{n-1}^M(F) \overset{\delta_{1}^{(n)}}{\larr}K_n^M(F)\arr 0
\]
with differentials
\[
\delta_i^{(n)}(a_1\otimes \cdots\otimes a_{i}\otimes\{b_1, \dots, b_{n-i}\})
\!=\!\!\sum_{j=1}^{i}a_1\otimes \cdots \otimes \widehat{a_j}\otimes \cdots\otimes a_{i}\otimes\{a_j, b_1, \dots, b_{n-i}\},
\]
is a complex.  If $n\geq 3$, it is easy to see that $\ker(\delta_1^{(n)})=\im(\delta_{2}^{(n)})$ (see Remark~\ref{rem:hh}). 

For any $n\geq 1$, let 
\[
\B_n(F):=
\begin{cases}
0 & \text{if $n=1$}\\
\widetilde{B}(F) & \text{if $n=2$,}\\
\ker(\delta_2^{(n)})/\im(\delta_{3}^{(n)}) & \text{if $n\geq 3$}
\end{cases}
\]
where $\widetilde{B}(F)$ is an extension of the Bloch group $B(F)$ by $\z/2$ if $\char(F)\!\neq\! 2$ and $\mu_{2^\infty}(F)$ is 
finite and is $B(F)$ otherwise. More precisely, $\widetilde{B}(F)\simeq K_3^\ind(F)/\tors(\mu(F), \mu(F))$, where the map 
$\widetilde{B}(F)\arr B(F)$ comes from the Bloch-Wigner exact sequence (\cite[Theorem~5.2]{suslin1991}, 
\cite[Theorem~5.1]{M-M2015}). Recall that $\mu(F)$ is the group of roots of unity and $\mu_{2^\infty}(F)$ is the group of
$2$-power roots of unity in $F$.

The groups $\B_n(F)$ seems to be interesting invariants of the field $F$.  Our first main result (Theorem \ref{thm:phi}) is:

\begin{athm} [{\bf Theorem A.}]
{\rm (i)} For  any positive integer $n$, there is a natural map 
{\[
\begin{array}{c}
\kappa_n\!\colon\!\B_n(F)\!\Big[\frac{1}{(n-2)!}\Big]\! \!\arr\! \ker\Big(\!H_{n}\!(\GL_{n-1}(F),\!\z\Big[\frac{1}{(n-2)!}\Big]\!)\!\! \arr\!\! 
H_{n}\!(\GL_{n}(F),\!\z\Big[\frac{1}{(n-2)!}\Big]\!)),
\end{array}
\]}
where its image is $(n-1)$-torsion.
\par {\rm (ii)} Let  the sequence
\[
\begin{array}{c}
H_n(\fff^2\times \GL_{n-2}(F),\zn)\overset{{\alpha_1}_\ast-{\alpha_2}_\ast}{-\!\!\!-\!\!\!-\!\!\!\larr} H_n(\fff\times\GL_{n-1}(F),\zn) \\
\overset{\inc_\ast}{\larr}H_n(\GL_{n}(F),\zn) \arr  0
\end{array}
\]
be exact for a fixed $n\geq 3$. If the natural map 
\[
\begin{array}{c}
H_m(\GL_{m-1}(F), \zmm) \arr H_m(\GL_{m}(F), \zmm)
\end{array}
\]
is injective for $m=n-1, n-2$, then $\kappa_n$ is surjective.
\end{athm}

We conjecture that $\kappa_n$ is surjective for any $n$ (Conjecture~\ref{conj:exact}) and prove it for
$n\leq 4$ (Corollary \ref{cor:n=4}). As an application we demonstrate that the natural homomorphisms
\[
\begin{array}{l}
H_{3}(\GL_{2}(F),\z\Big[\frac{1}{2}\Big])\arr H_{3}(\GL_{3}(F),\z\Big[\frac{1}{2}\Big]),
\end{array}
\]
\[
\begin{array}{l}
H_{4}(\GL_{3}(F),\z\Big[\frac{1}{6}\Big])\arr H_{4}(\GL_{4}(F),\z\Big[\frac{1}{6}\Big])
\end{array}
\]
are injective (Corollary \ref{cor:n=4}). The first injectivity (the case $n=3$) was already known by  
\cite[Theorem 5.4]{mirzaii-2008}. The second injectivity seems to be new.

Our second main result (Theorem \ref{thm:surj0}) concernes the quotient group
\[
\begin{array}{c}
H_{n+1}(\GL_{n}(F),\zn)/H_{n+1}(\GL_{n-1}(F),\zn).
\end{array}
\]
Inductively, the study of this group  can be reduced to the study of the quotient group 
\[
\begin{array}{c}
H_{n+1}(\GL_{n}(F),\zn)/H_{n+1}(\fff\times \GL_{n-1}(F),\zn).
\end{array}
\]
\begin{athm}[{\bf Theorem B.}]
Let 
\[
\begin{array}{c}
H_m(\GL_{m-1}(F), \zmm) \arr H_m(\GL_{m}(F), \zmm)
\end{array}
\]
be injective for $m=n-1, n-2$. Then there is a natural map
\[
\begin{array}{c}
\chi_n\colon \B_n(F)\nn \larr \displaystyle\frac{H_{n+1}( \GL_{n}(F),\z\nn)}{H_{n+1}(\fff\times \GL_{n-1}(F),\z\nn)}\cdot
\end{array}
\]
\end{athm}

We conjecture that $\chi_{n}$ always is surjective (Conjecture \ref{conj:chi}) and prove it for $n\leq 4$
(Propostion \ref{prp:n=2}, Corollary \ref{cor:B3-B4}). 

\begin{athm}[{\bf Remark.}]
{\rm The above results hold over any 
commutative ring with many units in the sense of Guin \cite[\S 1]{guin1989} (also see \cite[\S 2]{mirzaii2008}). 
The only exception is Proposition~\ref{prp:n=2}, where we need to assume that there is  a field $F$ and a 
ring homomorphism $R\arr F$ such that the restriction map $\mu(R)\arr \mu(F)$ is injective (see \cite[Theorem 5.1]{M-M2015}).
Recall that $R$ is a ring with many units if for any finite number of surjective linear
forms $f_i\colon R^2 \arr R$, there exists $v \in  R^2$ such that, for all $i$, $f_i(v) \in R^\times$. 
Important examples of rings with many units are semilocal rings with infinite residue fields.}
\end{athm}

In Section \ref{exa-Bn}, we study $\B_n(F)$ over certain fields and explain how the above theorems can be improved over them. 
Using these improved results for $n\leq 4$ and some {\sf unpublished results} of Galatius, Kupers and Randal-Williams (see
Theorems \ref{G-K-R}, \ref{g-k-r}) we prove the following (Theorem \ref{thm:div}):

\begin{athm}[{\bf Theorem C.}]
Let $F$ be a field such that $\fff$ is divisible. Then 
\par {\rm (i)} $H_{n}(\GL_{n-1}(F),\z) \arr H_{n}(\GL_{n}(F),\z)$ is injective for any $n$,
\par {\rm (ii)} $H_{n+1}(\GL_{n}(F),\z)/H_{n+1}(\GL_{n-1}(F),\z)$ is divisible for $n\neq 2$ and is uniquely divisible for $n\geq 5$.
\end{athm}

A field $F$ is called {\it real} if $-1$ is not the sum of squares and is called {\it real closed} if it is real and has no
real proper algebraic extension \cite[Chap. XI]{lang2002}. Over these fields we get (Theorem \ref{real}):

\begin{athm}[{\bf Theorem D.}]
Let  $F$ be a real closed field. Then
\par {\rm (i)}  $H_3(\GL_2(F),\z) \arr  H_3(\GL_3(F),\z)$ is injective,
\par {\rm (ii)}  $H_n(\GL_{n-1}(F),\z\half) \arr  H_n(\GL_n(F),\z\half)$ is injective for any $n$,
\par {\rm (iii)} $H_{n+1}(\GL_{n}(F),\z\half)/H_{n+1}(\GL_{n-1}(F),\z\half)$  is uniquely divisible for $n\geq 5$.
\end{athm}

\subsection*{Notation}
If $A \arr B$ is a homomorphism of abelian groups, by $B/A$ and $\im(A)$ we mean $\coker(A \arr B)$ and
$\im(A \arr B)$, respectively. We denote an element of $B/A$ 
represented by $b\in B$ again by $b$. Any inclusion of groups $H\se G$ is denoted by $\inc\colon H\arr G$. If $\kk$ is a commutative 
ring and $A$ an abelian group, by $A_\kk$ we mean $A \otimes_\z \kk$. Moreover, by $A\Big[\frac{1}{n}\Big]$ we mean 
$A\otimes _\z \z\Big[\frac{1}{n}\Big]=A_{\z[1/n]}$. We denote the $i$-th summand of $\fff^k:=\fff \times \cdots \times \fff$ by $F_i^\times$.

\subsection*{Convention}
 In this article we assume that $F$ always is an infinite field.

\subsection*{Acknowledgements.} 
The author thanks the anonymous referee for a very careful and thorough reading of the article and for very 
useful suggestions. Theorem C was suggested by the referee.

\section{{\bf The homology of general linear groups and Milnor \texorpdfstring{$K$}{Lg}-groups}}

For an arbitrary group $G$, let $B_\bullet(G)\overset{\varepsilon}{\arr}\z$ denote the right bar resolution of $G$.
For any left $G$-module $N$, $H_n(G, N)$ coincides  with the $n$-th homology of the complex 
$B_\bullet(G) \otimes_{\z[G]} N$. In particular,
\[
H_n(G,\z)=H_n(B_\bullet(G) \otimes_{\z[G]} \z)=H_n(B_\bullet(G)_G).
\]
For any $n$-tuple $(g_1, g_2, \dots, g_n)$ of pairwise commuting elements of $G$ and any commutative ring $\kk$, 
 let the homology class 
\[
{\rm \bf{c}}({g}_1, {g}_2,\dots, {g}_n)\in H_n(G,\kk)
\]
be represented by the cycle
\[
\sum_{\sigma \in \Sigma_n} {{\rm
sign}(\sigma)}[{g}_{\sigma(1)}| {g}_{\sigma(2)}|\dots|{g}_{\sigma(n)}]\otimes 1 \in B_n(G) \otimes_{\z[G]} \kk,
\]
where $\Sigma_n$ is the symmetric group of degree $n$.  In fact ${\rm \bf{c}}({g}_1,\dots, {g}_n)$ is
the image of $g_1\wedge \dots \wedge g_n$ under the composition
\[
\begin{array}{c}
\bigwedge_\z^n A_\kk \arr H_n(A,\kk) \arr H_n(G,\kk),
\end{array}
\]
where $A$ is the abelian subgroup of $G$ generated by $g_1,\dots, g_n$
and the first map is the Pontryagin product \cite[Chap. V, \S5 and \S6]{brown1994}.
It follows immediately from the known properties of the Pontryagin product that:\\
\par {\rm (a)} If $h_1\in G$ commutes with all the elements
$g_1, \dots, g_n$, then
\[
{\rm \bf{c}}(g_1h_1, g_2,\dots, g_n)= {\rm \bf{c}}(g_1, g_2,\dots,
g_n)+{\rm \bf{c}}(h_1, g_2,\dots, g_n).
\]
\par\ \ \ \ In particular if $h_1=g_1^{-1}$, then we have
\[
{\rm \bf{c}}(g_1^{-1}, g_2,\dots, g_n)=- {\rm \bf{c}}(g_1, g_2,\dots,g_n).
\]
\par {\rm (b)} For every $\sigma \in \Sigma_n$, 
\[
{\rm \bf{c}}(g_{\sigma(1)},\dots, g_{\sigma(n)})={\rm sign(\sigma)} {\rm \bf{c}}(g_1,\dots,g_n).
\]
\par {\rm (c)} The cup product of ${\rm \bf{c}}(g_1,\dots, g_p)\!\in\! H_p(G,\kk)$
and ${\rm \bf{c}}(g_1',\dots, g_q')\! \in \! H_q(G',\kk)$ is 
\[
{\rm\bf{c}}((g_1, 1), \dots, (g_p,1),(1,g_1'), \dots, (1,g_q')) \in H_{p+q}(G \times G',\kk).
\]

Let $F$ be an infinite field. For any $a\in \fff$ and any $1\leq i\leq n$, let  $\DD_{i, n}(a)$ be the diagonal matrix of size $n$  
with $a$ in the $i$-th position  of the diagonal and $1$ everywhere else. It is well-known (see \cite[Lemma 2.7.1]{suslin1985} or 
\cite[Proposition 3.1.2]{guin1989}) that the map 
\begin{gather*}
c_n\colon \fff \otimes_\z \cdots \otimes_\z \fff \arr H_n(\GL_n(F),\z)/H_n(\GL_{n-1}(F),\z),\\
a_1\otimes \cdots \otimes a_n \mt  {\bf c}(\DD_{1, n}(a_1), \dots, \DD_{n, n}(a_n))\!\!\!\!\pmod {H_n(\GL_{n-1}(F),\z)}
\end{gather*}
factors through the $n$-th Milnor $K$-group  of $F$:
\[
\bar{c}_n\colon K_n^M(F) \arr H_n(\GL_n(F),\z)/H_n(\GL_{n-1}(F),\z).
\]
Observe that $c_n$ is nothing but the composite
\begin{align*}
H_1(\fff,\z)^{\otimes n}\overset{\cup}{\larr}  H_n(T_n(F), \z) &\arr H_n(\GL_n(F),\z) \\
&\two H_n(\GL_n(F),\z)/H_n(\GL_{n-1}(F),\z),
\end{align*}
with $T_n(F)$ the group of diagonal $(n\times n)$-matrices.

\begin{thm}[Suslin, Nesterenko-Suslin]\label{suslin}
{\rm (i)} The natural map 
\[
\inc_\ast\colon H_n(\GL_r(F),\z)\arr H_n(\GL_{r+1}(F),\z)
\] 
is an isomorphism for any $r \geq n$.
\par {\rm (ii)} 
There exists a natural map  $s_n\colon H_n(\GL_n(F),\z) \arr K_n^M(F)$ such that 
\[
H_n(\GL_{n-1}(F),\z) \overset{\inc_\ast}{\larr} H_n(\GL_n(F),\z) \overset{s_n}{\larr} K_n^M(F) \arr 0
\]
is an exact sequence. Moreover, $\bar{s}_n$ is the inverse of $\bar{c}_n$.
\par {\rm (iii)} The composite 
\[
K_n^M(F) \arr K_n(F)\overset{h_n}{\larr} H_n(\GL(F),\z)\overset{\inc_\ast^{-1}}{\larr} H_n(\GL_n(F),\z)\overset{s_n}{\larr} K_n^M(F)
\]
coincides with multiplication by $(-1)^{n-1}(n-1)!$, where $h_n$ is the Hurewicz map.
\end{thm}
\begin{proof}
See \cite[Theorem 3.4, Corollary 4.4]{suslin1985} or \cite[Theorems 3.25, 4.1, Remark 3.27]{N-S1990}.
\end{proof}

In  \cite[Proposition~4]{mirzaii2008} the author showed that if $n\geq 3$ and if $\kk$ is a field such that $(n-1)!\in \kk^\times$, 
then the injectivity of 
\[
H_n(\GL_{n-1}(F),\kk) \arr H_n(\GL_n(F),\kk)
\]
 follows, by an induction process, from the exactness of the complex 
\begin{equation}\label{exactness}
H_n(\fff^2\times \GL_{n-2}(F),\kk) \overset{{\alpha_1}_\ast-{\alpha_2}_\ast}{-\!\!\!-\!\!\!-\!\!\!-\!\!\!\larr} 
H_n(\fff \times \GL_{n-1}(F),\kk) 
\end{equation}
\[
\overset{\inc_\ast}{\larr} H_n(\GL_n(F),\kk) \arr 0,
\]
where $\alpha_1(\diag(a,b, A))=\diag(b,a,A)$ and $\alpha_2=\inc$. 

The exactness of the above complex was known for $n=3$ \cite[Corollary~3.5]{mirzaii-2008} and $n=4$ \cite[Theorem~2]{mirzaii2008}. 
Recently Galatius, Kupers and Randal-Williams have proved, by completely new and interesting methods, that this complex is indeed 
exact.  Thus they proved Suslin's injectivity conjecture for $r=n-1$.

\begin{thm}[\text{Galatius, Kupers, Randal-Williams \cite[Section~9.6]{G-K-R2020}}] \label{G-K-R}
Let $\kk$ be a field such that $(n-1)!\in \kk^\times$. Then
\par {\rm (i)} For any $n\geq 3$, the complex $(\ref{exactness})$ is exact,
\par {\rm (ii)} For any $n$,  the natural map $H_n(\GL_{n-1}(F),\kk)\arr H_n(\GL_n(F),\kk)$ is injective.
\end{thm}

\begin{rem}\label{conj:inj}
We conjecture that for any $n$, the natural map
\[
\begin{array}{c}
H_n(\GL_{n-1}(F),\znn) \arr H_n(\GL_{n}(F),\znn)
\end{array}
\]
is injective (see also Remark \ref{rem:k3}(ii) bellow). This follows from the exactness of the sequence
\[
\begin{array}{c}
H_n(\fff^2\times \GL_{n-2}(F),\znn)\overset{{\alpha_1}_\ast-{\alpha_2}_\ast}{-\!\!\!-\!\!\!-\!\!\!-\!\!\!\larr}
H_n(\fff\times\GL_{n-1}(F),\znn)\\
\overset{\inc_\ast}{\larr} H_n(\GL_{n}(F),\znn) \arr   0
\end{array}
\]
for all $n\geq 3$ (see the proof of Theorem \ref{thm:phi} below). The exactness of the above sequence  
does not follow directly from  the result of Galatius, Kupers, Randal-Williams mentioned in Theorem \ref{G-K-R} (at least we couldn't prove this).
But we are very hopeful that their methods developed in \cite[Section 9]{G-K-R2020} might be applied to this general setting.
\end{rem}

\section{On the kernel of \texorpdfstring{$H_{n}(\GL_{n-1}(F),\z) \arr H_{n}(\GL_{n}(F),\z)$}{Lg}}

In this section we show that the group $\B_n(F)$ is related to the kernel of 
\[
H_{n}(\GL_{n-1}(F),\z) \arr H_{n}(\GL_{n}(F),\z).
\]
This connection first was realised for $n=3$ in  \cite[Remark~3.5]{mirzaii2012}. The following is Theorem A of the introduction.

\begin{thm}\label{thm:phi}
\par {\rm (i)} For  any positive integer $n$, there is a natural map 
\[
\varphi_n\colon\B_{n}(F) \arr \ker\Big(H_{n}(\GL_{n-1}(F),\z) \arr H_{n}(\GL_{n}(F),\z)\Big),
\]
which factors through the multiplication by $(n-2)$-map $\B_n(F)\arr (n-2)\B_{n}(F)$, $x\mt (n-2)x$, and its image is $(n-1)$-torsion.
In particular, there is a natural map 
\[
\begin{array}{c}
\kappa_n\colon\!\B_n(F)\nn\!\arr\! \ker\Big(H_n(\GL_{n-1}(F),\zn)\! \arr\! 
H_n(\GL_{n}(F),\zn)\Big),
\end{array}
\]
where  its image is $(n-1)$-torsion.
\par {\rm (ii)} Let  the sequence
\[
\begin{array}{c}
H_n(\fff^2\times \GL_{n-2}(F),\zn)\overset{{\alpha_1}_\ast-{\alpha_2}_\ast}{-\!\!\!-\!\!\!-\!\!\!\larr} H_n(\fff\times\GL_{n-1}(F),\zn) \\
\overset{\inc_\ast}{\larr}H_n(\GL_{n}(F),\zn) \arr  0
\end{array}
\]
be exact for a fixed $n\geq 3$. If the natural map $H_m(\GL_{m-1}(F), \zmm) \arr H_m(\GL_{m}(F), \zmm)$ is
injective for $m=n-1, n-2$, then $\kappa_n$ is surjective.
\end{thm}
\begin{proof}
(i) We define $\varphi_1$ and $\varphi_2$ as trivial maps. So let $n\geq 3$.
Let $x=\sum a\otimes b\otimes\{c_1, \dots, c_{n-2}\}\in \fff \otimes \fff\! \otimes K_{n-2}^M(F)$ 
represents an element of $\B_{n}(F)$. Let $y$ be the image of $x$ by the map  
\[
\id_\fff\!\otimes\id_\fff\!\otimes\iota_{n-2}\colon\! \fff\! \otimes \fff\! \otimes\! 
K_{n-2}^M(F)\! \arr\! F^\times\! \otimes F^\times\! \otimes\! H_{n-2}(\GL_{n-2}(F),\z).
\]
Thus 
\[
y=\sum a\otimes b\otimes[c_1,\dots,c_{n-2}]=\sum a\otimes b\otimes {\bf c}(C_{1,n-2}, \dots, C_{n-2,n-2}). 
\]
Consider the complex
\[
\begin{array}{c}
H_n(\fff^2\times \GL_{n-2}(F),\z) \overset{{\alpha_1}_\ast-{\alpha_2}_\ast}{-\!\!\!-\!\!\!-\!\!\!-\!\!\!\larr} 
H_n(\fff \times \GL_{n-1}(F),\z) \\
\overset{\inc_\ast}{\larr} H_n(\GL_n(F),\z) \arr 0,
\end{array}
\]
and set $\alpha={\alpha_1}_\ast-{\alpha_2}_\ast$.
The restriction of $\alpha$ on 
\[
F_1^\times\otimes F_2^\times\otimes H_{n-2}(\GL_{n-2}(F),\z)\se H_n(\fff^2\times \GL_{n-2}(F),\z)
\]
factors through
$F^\times\otimes  H_{n-1}(\GL_{n-1}(F),\z)\se H_n(\fff \times \GL_{n-1}(F),\z)$ and we have
\[
\begin{array}{l}
\alpha(y)\!=\!-\!\sum\! \Big(b\otimes {\bf c}(\diag(a, I_{n-2}), \diag(1, C_{1,n-2})), \dots, \diag(1, C_{n-2,n-2}))\\
\ \ \ \ \ \ \ \ \ \ \ \ \ \ \ \ \  \!+a\otimes {\bf c}(\diag(b, I_{n-2}), \diag(1, C_{1,n-2}), \dots, \diag(1, C_{n-2,n-2}))\Big)\\
\ \ \ \ \ \ \ \ \  =\!-\!\sum\! \Big(b\otimes {\bf c}(\diag(I_{n-2}, a), C_{1,n-1}, \dots, C_{n-2,n-1})\\
\ \ \ \ \ \ \ \ \ \ \ \ \ \ \ \ \  \!+a\otimes {\bf c}(\diag(I_{n-2},b), C_{1,n-1}, \dots, C_{n-2,n-1})\Big).
\end{array}
\]
(Here we denote the $i$-th summand of  $\fff^2=\fff \times \fff$ by $F_i^\times$ for $i=1,2$).
On the other hand, the composite 
\[
F^\times\otimes F^\times\otimes\! K_{n-2}^M(F)\! \overset{\delta_2^{(n)}}{-\!\!\!\larr} \! F^\times\otimes K_{n-1}^M(F)
\!\!\overset{\id_\fff\otimes \iota_{n-1}}{-\!\!\!-\!\!\!-\!\!\!-\!\!\!-\!\!\!\larr}\! F^\times\otimes H_{n-1}(\GL_{n-1}(F),\z)
\]
takes $x$ to $0=\sum \Big(b\otimes [a, c_1,\dots, c_{n-2}]+a\otimes [b, c_1,\dots, c_{n-2}]\Big)$. We have
\[
\begin{array}{rl}
\!\!\alpha\!((n-2)y)&\!\!\!\!\!=\!-(n-2)\!\sum\! \Big(b\otimes {\bf c}(\diag(I_{n-2}, a), C_{1,n-1}, \dots, C_{n-2,n-1})\\
& \ \ \ \ \ \ \ \ \ \ \ \ \ \   +a\otimes {\bf c}(\diag(I_{n-2},b), C_{1,n-1}, \dots, C_{n-2,n-1})\Big)\\
&\!\!\!\!=(\!-1)^{n-2}\!\sum\! \Big(b\otimes {\bf c}(C_{1,n-1}, \dots, C_{n-2,n-1}, \diag(I_{n-2}, a^{-(n-2)}))\\
&\hspace{1.7cm}+a\otimes\! {\bf c}(C_{1,n-1}, \dots, C_{n-2,n-1}, \diag(I_{n-2}, b^{-(n-2)}))\Big)\\
&\!\!\!\!=(-1)^{n-2}\!\sum\! \Big(b \otimes [c_1, \dots, c_{n-2},a] + a \otimes [c_1, \dots, c_{n-2},b]\\
&\hspace{1.8cm} -b \otimes{\bf c}(C_{1, n-1}, \dots ,C_{n-2, n-1}, \diag(aI_{n-2}, 1))\\
& \hspace{1.8cm}-a \otimes{\bf c}(C_{1, n-1}, \dots ,C_{n-2, n-1}, \diag(bI_{n-2}, 1))\Big)\\
&\!\!\!\!= -\sum \Big(b \otimes{\bf c}(\diag(aI_{n-2}, 1),C_{1, n-1}, \dots ,C_{n-2, n-1})\\
&\hspace{0.9cm} +a \otimes{\bf c}(\diag(bI_{n-2}, 1),C_{1, n-1}, \dots ,C_{n-2, n-1})\Big).
\end{array}
\]
Let $A_1$, $A_2$ and $B_1$, $B_2$ be the following summands of $H_n(\fff \times \GL_{n-1}(F),\z)$ and 
$H_n(\fff^2\times \GL_{n-2}(F),\z)$, respectively:  
\[
\begin{array}{ll}
A_1 =H_{n}(\GL_{n-1}(F),\z),& B_1=F_2^\times\otimes H_{n-1}(\GL_{n-2}(F),\z), \\
A_2=F^\times\otimes H_{n-1}(\GL_{n-1}(F),\z), &B_2=F_1^\times\otimes F_2^\times\otimes H_{n-2}(\GL_{n-2}(F),\z).
\end{array}
\]
Consider the following element of $B_1$:
\[
\begin{array}{l}
z=\sum \Big(b \otimes{\bf c}(aI_{n-2}, C_{1, n-2}, \dots ,C_{n-2, n-2}) \\
\ \ \ \ \ \ \ \  +a \otimes{\bf c}(bI_{n-2}, C_{1, n-2}, \dots, C_{n-2, n-2})\Big).
\end{array}
\]
The map $\alpha|_{B_1\oplus B_2}$ factors through $A_1\oplus A_2$ and 
\[
\alpha|_{B_1\oplus B_2}(z,(n-2)y)=(-w,0)\in A_1\oplus A_2\se H_n(\fff \times \GL_{n-1}(F),\z),
\]
 where
\[
\begin{array}{l}
w=\sum\Big( {\bf c}(\diag(I_{n-2}, b), \diag(aI_{n-2},1), C_{1, n-1}, \dots ,C_{n-2, n-1})\\
\ \ \ \ \ \ \ \ \ \ \ \ + {\bf c}(\diag(I_{n-2},a), \diag(bI_{n-2},1),  C_{1, n-1}, \dots ,C_{n-2, n-1})\Big).
\end{array}
\]
Clearly $w$ is in the kernel of $H_{n}(\GL_{n-1}(F),\z) \arr H_{n}(\GL_{n}(F),\z)$. Now define 
\[
\varphi_n\colon\B_{n}(F) \arr \ker\Big(H_{n}(\GL_{n-1}(F),\z) \arr H_{n}(\GL_{n}(F),\z)\Big), \ \ \ \varphi(x)=w.
\]
By a direct computation, using the definition of $\varphi$, it is easy to see that 
\[
\varphi_n\Big(b\otimes c\otimes  \{d_1, \dots, d_{n-3}, a\}+a\otimes c \otimes \{d_1, \dots, d_{n-3}, b\}
+a\otimes b\otimes \{d_1, \dots, d_{n-3}, c\}\Big)
\]
is trivial. This shows that $\varphi_n$ is well-defined.

To prove that $w$ is $(n-1)$-torsion, consider the composite
\begin{gather*}
\fff\! \otimes\! K_{n-1}^M(F)
\!\!\overset{\id_\fff\!\otimes \iota_{n-1}}{-\!\!\!-\!\!\!-\!\!\!-\!\!\!\larr} \! F^\times\! \otimes\! H_{n-1}(\GL_{n-1}(F),\z)\! 
\overset{\cup}{\arr}\! H_n(\fff\! \times \GL_{n-1}(F),\z) \\
\overset{\tau_\ast}{\larr} H_n(\GL_{n-1}(F),\z),
\end{gather*}
where $\tau \colon\fff \times \GL_{n-1}(F)\arr \GL_{n-1}(F)$ is the multiplication $(a, A)\mt aA=(aI_{n-1})A$. Under this composite,
\[
\begin{array}{rl}
0&=\sum \Big(b\otimes \{a, c_1,\dots, c_{n-2}\}+a\otimes \{b, c_1,\dots, c_{n-2}\}\Big)\\
&=(-1)^{n-2}\sum \Big(b\otimes \{c_1,\dots, c_{n-2},a\}+a\otimes \{c_1,\dots, c_{n-2}, b\}\Big)
\end{array}
\]
 maps to
\[
\begin{array}{rl}
0&\!\!\! =\!(-1)^{n-2}\!\sum\! \Big({\bf c}(bI_{n-1}, C_{1,n-1}, \dots, C_{n-2,n-1}, A_{n-1, n-1}) \\
& \hspace{1.8cm} + {\bf c}(aI_{n-1}, C_{1,n-1}, \dots, C_{n-2,n-1}, B_{n-1, n-1})\Big)\\
& \!\!\!=\sum \Big({\bf c}(\diag(bI_{n-2}, b), \diag(aI_{n-2}, a^{-(n-2)}), C_{1,n-1}, \dots, C_{n-2,n-1}) \\
&\ \ \ \ \!\ +  {\bf c}(\diag(aI_{n-2}, a), \diag(bI_{n-2}, b^{-(n-2)}), C_{1,n-1}, \dots, C_{n-2,n-1})\Big)\\
&\!\!\!=\sum \Big({\bf c}(\diag(bI_{n-2}, 1), \diag(I_{n-2}, a^{-(n-2)}),C_{1,n-1}, \dots,C_{n-2,n-1}) \\
&\hspace{0.6cm}+ {\bf c}(\diag(I_{n-2}, b), \diag(aI_{n-2}, 1), C_{1,n-1}, \dots, C_{n-2,n-1}) \\
&\hspace{0.6cm} +{\bf c}(\diag(aI_{n-2}, 1), \diag(I_{n-2}, b^{-(n-2)}), C_{1,n-1}, \dots, C_{n-2,n-1}) \\
& \hspace{0.6cm}+ {\bf c}(\diag(I_{n-2}, a), \diag(bI_{n-2}, 1),C_{1,n-1}, \dots,C_{n-2,n-1})\Big)\\
& \!\!\!=(n-1)w.
\end{array}
\]
(Observe that $C_{i,n-1}=\diag(C_{i,n-2}, 1)$ for $1\leq i\leq n-2$.)
Therefore $w$ is $(n-1)$-torsion. 

The second part follows from tensoring $\varphi_n$ with $\zn$ and 
the fact that localization is exact. Following the above argument, it is easy to see that
for any $n\geq 3$,  $\kappa_n$ is given by
\[
\sum a\otimes b\otimes\{c_1, \dots, c_{n-2}\} \mt 
\begin{array}{c}
\frac{(-1)^{n-2}}{(n-2)!}w.
\end{array}
\]
\par (ii) The claim is trivial for $n=1,2$. So let $n\geq 3$. By Corollary \ref{cor:dec} and  the injectivity hypothesis for  $m=n-1, n-2$, we have 
 the decomposition
\[
\begin{array}{c}
H_m(\GL_{m}(F),\zmm)\!\simeq\! H_{m}(\GL_{m-1}(F),\zmm)\!\oplus\! K_m^M(F)\Big[\frac{1}{(m-1)!}\Big].
\end{array}
\]

For simplicity, we set $\kk=\zn$. By the K\"unneth formula, and specifying a splitting (see \cite[Theorem 2.7]{taylor2015}), we have
\[
\begin{array}{c}
H_n(\fff\times \GL_{n-1}(F),\kk)\simeq \bigoplus_{i=1}^4 T_i, \\
\ H_n(\fff^2\times \GL_{n-2}(F),\kk)\simeq \bigoplus_{i=1}^{11} U_i,
\end{array}
\]
where
\[
\hspace{-1cm}
\begin{array}{l}
T_1=H_n(\GL_{n-1}(F),\kk),\\
T_2=\fff\otimes H_{n-1}(\GL_{n-1}(F),\kk)\simeq T_2 '\oplus T_2'', \\
T_2'=\fff\otimes H_{n-1}(\GL_{n-2}(F),\kk),\\
T_2''=\fff\otimes K_{n-1}^M(F)_\kk, \\
T_3=\bigoplus_{i=2}^n H_i(\fff,\kk)\otimes H_{n-i}(\GL_{n-1}(F),\kk),\\
T_4=\bigoplus_{i=1}^{n-2}\tors(H_i(F^\times,\z), H_{n-i-1}(\GL_{n-1}(F),\z))_\kk,
\end{array}
\]
and 
\[
\begin{array}{l}
U_1 =H_n(\GL_{n-2}(F),\kk),\\
U_2 = \bigoplus_{i=1}^n H_i(F_1^\times,\kk) \otimes H_{n-i}(\GL_{n-2}(F),\kk),\\
U_3 = \bigoplus_{i=1}^n H_i(F_2^\times,\kk) \otimes H_{n-i}(\GL_{n-2}(F),\kk),\\
U_4 = F_1^\times\otimes F_2^\times \otimes H_{n-2}(\GL_{n-2}(F),\kk)\simeq U_4'\oplus U_4'',\\
U_4' =F_1^\times\otimes F_2^\times \otimes H_{n-2}(\GL_{n-3}(F),\kk),\\
U_4''=F_1^\times\otimes F_2^\times \otimes K_{n-2}^M(F)_\kk,\\
U_5 = \bigoplus_{\underset{i,j>0}{i+j\geq 3}} H_i(F_1^\times,\kk) \otimes H_j(F_2^\times,\kk) \otimes H_{n-i-j}(\GL_{n-2}(F),\kk),\\
U_6 = \bigoplus_{i=1}^{n-2}\tors(H_i(F_1^\times,\z), H_{n-i-1}(\GL_{n-2}(F),\z))_\kk,\\
U_7 = \bigoplus_{i=1}^{n-2}\tors(H_i(F_2^\times,\z), H_{n-i-1}(\GL_{n-2}(F),\z))_\kk,\\
U_8 = \bigoplus_{i=1}^{n-2}\tors(H_i(F_1^\times,\z), H_{n-i-1}(F_2^\times,\z))_\kk.\\
U_9=\bigoplus_{\underset{i,j>0}{i+j\leq n-2}}H_i(F_1^\times,\z)\otimes \tors\Big(H_j(F_2^\times,\z), H_{n-i-j-1}(\GL_{n-2}(F),\z)\Big)_\kk,\\
U_{10}=
\bigoplus_{\underset{i,j>0}{i+j\leq n-2}}\tors\Big(H_i(F_1^\times,\z), H_j(F_2^\times,\z)\otimes H_{n-i-j-1}(\GL_{n-2}(F),\z)\Big)_\kk,\\
U_{11}=
\bigoplus_{\underset{i,j>0}{i+j\leq n-2}}\tors\bigg(H_i(F_1^\times,\z), \tors\big(H_j(F_2^\times,\z), H_{n-i-j-2}(\GL_{n-2}(F),\z)\Big)\bigg)_\kk.\\
\end{array}
\]

Let $t_1\in \ker \Big(H_n(\GL_{n-1}(F),\kk) \arr H_{n}(\GL_n(F), \kk)\Big)$. Then $t=(t_1,0,0,0)$ is in the kernel of 
\[
\begin{array}{l}
\inc_\ast\colon H_n(\fff\times \GL_{n-1}(F),\kk) \arr H_n(\GL_{n}(F),\kk).
\end{array}
\]
Thus there exists $u=(u_1, \dots, u_{11})\in H_n(\fff^2\times \GL_{n-2}(F),\kk)$ such that 
\[
\alpha(u)=t,
\]
where $\alpha:={\alpha_1}_\ast-{\alpha_2}_\ast$. Consider the complex
\[
\begin{array}{c}
\hspace{-3.3cm}
H_n(\fff\! I_2\times \GL_{n-2}(F),\kk)\oplus H_n(\fff^3\times \GL_{n-3}(F),\kk) \\
\hspace{2.5cm}
 \overset{\beta}{\larr} H_n(\fff^2\times \GL_{n-2}(F),\kk) \overset{\alpha}{\larr} H_n(\fff \times \GL_{n-1}(F),\kk),
\end{array}
\]
where $\fff I_2$ is the group of the non-zero multiples of the $(2\times 2)$ identity matrix and
\[
\beta=(\inc_\ast, {\sigma_1}_\ast-{\sigma_2}_\ast+{\sigma_3}_\ast)
\]
 with
$\diag(a,b,c, A)\overset{\sigma_1}{\mt} \diag(b,c, a, A)$,  $\diag(a,b,c, A)\overset{\sigma_2}{\mt} \diag(a,c, b, A)$ and 
$\diag(a,b,c, A)\overset{\sigma_3=\inc}{\mt} \diag(a,b,c, A)$.

Let $H_n(\fff I_2\times \GL_{n-2}(F),\kk)\simeq V_1\oplus V_2\oplus V_3$, where
\[
\begin{array}{l}
V_1  =H_n(\GL_{n-2}(F),\kk),\\
V_2  = \bigoplus_{i=1}^n H_i(F^\times I_2,\kk) \otimes H_{n-i}(\GL_{n-2}(F),\kk),\\
V_3 = \bigoplus_{i=1}^{n-2}\tors(H_i(F^\times I_2,\z), H_{n-i-1}(\GL_{n-2}(F),\z))_\kk.
\end{array}
\]
Since $\beta|_{H_n(\fff I_2\times \GL_{n-2}(F),\kk)}=\inc_\ast$, we have
\[
\beta((u_1,u_2,u_6), 0)=(u_1, u_2, u_2, 0,0, u_6,u_6, 0, 0, 0, 0).
\]
Observe that
\[
u-\beta((u_1,u_2,u_6), 0)=(0,0, u_3-u_2, u_4,u_5, 0, u_7-u_6, u_8,u_9, u_{10}, u_{11})
\]
and 
\[
\alpha(u-\beta((u_1,u_2,u_6), 0))=t.
\]
Thus  from the beginning we may assume that $u_1=u_2=u_6=0$. 

Let $W$ be the following summand of $H_n(\fff^3\times \GL_{n-3}(F),\kk)$:
\[
W=\bigoplus_{\underset{i,j>0}{i+j\geq 3}}W_{i,j}
=\bigoplus_{\underset{i,j>0}{i+j\geq 3}} H_i(F_2^\times,\kk) \otimes H_j(F_3^\times,\kk) \otimes H_{n-i-j}(\GL_{n-3}(F),\kk).
\]
The restriction of $\beta$ on $W_{i,j}$ factors through $U_3\oplus U_5$ and 
\begin{align*}
\beta|_{W_{i,j} }\colon W_{i,j} &\arr U_3\oplus U_5 \se H_n(\fff^2\times \GL_{n-2}(F),\kk), \\
x\otimes y\otimes z& \mt \Big((-{\sigma_{2}}_\ast+{\sigma_{3}}_\ast)(x\otimes y\otimes z), x\otimes y\otimes \inc_\ast(z)\Big).
\end{align*}
By homological stability, $H_{n-i-j}(\GL_{n-3}(F),\kk)\arr H_{n-i-j}(\GL_{n-2}(F),\kk)$ is an isomorphism for $i+j\geq 3$. 
Thus we may assume $u_5=0$ (similar to the elimination of  $u_1,u_2, u_6$).

The restriction of $\beta$ on the  summand $W'=F_2^\times\otimes F_3^\times\otimes H_{n-2}(\GL_{n-3}(F),\kk)$
of $H_n(\fff^3\times \GL_{n-3}(F),\kk)$ factors through $U_3\oplus U_4'$ and
\begin{align*}
\beta|_{W'}\colon W' &\arr U_3\oplus U_4'\se H_n(\fff^2\times \GL_{n-2}(F),\kk), \\
a\otimes b\otimes z &\mt \Big((-{\sigma_{2}}_\ast+{\sigma_{3}}_\ast)(a\otimes b\otimes z),a\otimes b\otimes z\Big).
\end{align*}
Thus we may assume that $u_4'=0$. The restriction of $\beta$ on the summand 
\[
\begin{array}{c}
W''=\bigoplus_{i=1}^{n-2}\tors(H_i(F_2^\times,\z), H_{n-i-1}(F_3^\times,\z))_\kk
\end{array}
\]
of $H_n(\fff^3\times \GL_{n-3}(F),\kk)$ factors through $U_7\oplus U_8$ and
\[
\begin{array}{c}
\beta|_{W''}\colon W'' \arr U_7\oplus U_8\se H_n(\fff^2\times \GL_{n-2}(F),\kk), \\
\hspace{-2.5 cm}
x\mt ((-{\sigma_2}_\ast+{\sigma_3}_\ast) (x), x).
\end{array}
\]
Hence we may assume $u_8=0$.  If $X$, $X'$ and $X''$ are the following summands 
\[
\begin{array}{l}
X =\bigoplus_{\underset{i,j>0}{i+j\leq n-2}}
H_i(F_2^\times,\z)\otimes \tors\Big(H_j(F_3^\times,\z), H_{n-i-j-1}(\GL_{n-3}(F),\z)\Big)_\kk,\\
X'=\bigoplus_{\underset{i,j>0}{i+j\leq n-2}}
\tors\Big(H_i(F_2^\times,\z), H_j(F_3^\times,\z)\otimes H_{n-i-j-1}(\GL_{n-3}(F),\z)\Big)_\kk\\
X''=\bigoplus_{\underset{i,j>0}{i+j\leq n-2}}
\tors\bigg(H_i(F_2^\times,\z), \tors\big(H_j(F_3^\times,\z), H_{n-i-j-2}(\GL_{n-3}(F),\z)\Big)\bigg)_\kk
\end{array}
\]
of $H_n(\fff^3\times \GL_{n-3}(F),\kk)$, then 
\[
\begin{array}{ll}
\beta|_X\colon X \arr U_3\oplus U_7\oplus U_9, & x\mt (\sigma_3(x),-\sigma_2(x), \inc_\ast(x)),\\
\beta|_{X'}\colon X' \arr U_7\oplus U_{10},&  y \mt ((-{\sigma_{2}}_\ast+{\sigma_{3}}_\ast)(y), \inc_\ast(y), \\
\beta|_{X''}\colon X'' \arr U_7\oplus U_{11}, &  z \mt ((-{\sigma_{2}}_\ast+{\sigma_{3}}_\ast)(z), \inc_\ast(z).
\end{array}
\]
Now by homological stability we may assume that $u_9=u_{10}=u_{11}=0$.
Thus $u$ finds the following form
\[
u=(u_{3}, u_4'', u_7) \in U_{3}\oplus U_4''\oplus U_7\se H_n(\fff^2\times \GL_{n-2}(F),\kk).
\]

Let $U_3=\bigoplus_{i=1}^n U_{3, i}$ and $u_3=(u_{3,i})_{1\leq i \leq n}$. If $T_3=\bigoplus_{i=1}^n T_{3, i}$, then
$\alpha$ factors as follow on the following summands:
\[
\begin{array}{l}
\alpha|_{U_{3, 1}}\colon  U_{3, 1} \arr T_1\oplus T_2' \se H_n(\fff\times \GL_{n-1}(F),\kk), \\
\hspace{2.4 cm} s\otimes z \mt (-s\cup z, s\otimes z),\\
\alpha|_{U_{3, i}}\colon U_{3, i} \arr T_1\oplus T_{3,i}\se H_n(\fff\times \GL_{n-1}(F),\kk),  \  2\leq i \leq n,\\
\hspace{2.3 cm} r\otimes v \mt (-r\cup v, r\otimes \inc_\ast(v)),\\
\alpha|_{U_7}\colon  U_7\arr T_1\oplus T_4\se H_n(\fff\times \GL_{n-1}(F),\kk), \\
\hspace{2.4 cm} u_7\mt (-{\alpha_2}_\ast(u_7), {\alpha_1}_\ast(u_7)).
\end{array}
\]
Moreover, the restriction of  $\alpha$ on $U_4''$  factors through  $T_2=T_2'\oplus T_2''$. More precisely, we have 
\[
\begin{array}{c}
\alpha|_{U_4''}\colon U_4'' \arr T_2=T_2'\oplus T_2''\se H_n(\fff\times \GL_{n-1}(F),\kk), \\
a\otimes b\otimes \{c_1,\dots, c_{n-2}\}\mt t=(t', t''),
\end{array}
\]
where 
\[
\begin{array}{rl}
t &\!\!\!\!=-\frac{(-1)^{n-3}}{(n-3)!}
\Big(b \otimes {\bf c}(\diag(a, I_{n-2}),\diag(1, C_{1, n-2}), \dots ,\diag(1, C_{n-2, n-2}))\\
 &\ \ \ \ \ \ \ \ \ \ \  +a \otimes {\bf c}(\diag(b, I_{n-2}),\diag(1, C_{1, n-2}), \dots ,\diag(1, C_{n-2, n-2}))\Big)\\
&\!\!\!\!=-\frac{-(n-2)}{(n-2)!}
\Big(b \otimes {\bf c}(\diag(C_{1, n-2}, 1), \dots ,\diag(C_{n-2, n-2}, 1), \diag(I_{n-2}), a)\\
 &\ \ \ \ \ \ \ \ \ \ \  +a \otimes {\bf c}(\diag(C_{1, n-2}, 1), \dots ,\diag(C_{n-2, n-2}, 1), \diag(I_{n-2}), b)\Big)\\
 &\!\!\!\!=-\frac{1}{(n-2)!}
\Big(b \otimes {\bf c}(\diag(C_{1, n-2}, 1), \dots ,\diag(C_{n-2, n-2}, 1), \diag(I_{n-2}, a^{-(n-2)}))\\
 &\ \ \ \ \ \ \ \ \ \  +a \otimes {\bf c}(\diag(C_{1, n-2}, 1), \dots ,\diag(C_{n-2, n-2}, 1), \diag(I_{n-2}, b^{-(n-2)}))\Big)\\
&\!\!\!\! =-\frac{1}{(n-2)!}\Big(b \otimes [c_1, \dots, c_{n-2},a] + a \otimes [c_1, \dots, c_{n-2},b]\\
&\ \ \ \ \ \ \ \  -b \otimes{\bf c}(\diag(C_{1, n-2}, 1), \dots ,\diag(C_{n-2, n-2}, 1), \diag(aI_{n-2}, 1))\\
&\ \ \ \ \ \ \ \ -a \otimes{\bf c}(\diag(C_{1, n-2}, 1), \dots ,\diag(C_{n-2, n-2}, 1), \diag(bI_{n-2}, 1))\Big)\\
t' &\!\!\!\!=\frac{(-1)^{n-2}}{(n-2)!}\inc_\ast\Big(b \otimes{\bf c}(aI_{n-2}, C_{1, n-2}, \dots, C_{n-2, n-2})\\
&\ \ \ \ \ \ \ \ \ \ \ \ +a \otimes{\bf c}(bI_{n-2}, C_{1, n-2}, \dots, C_{n-2, n-2})\Big),\\
t''&\!\!\!\!= -b \otimes \{a, c_1, \dots, c_{n-2}\} -a \otimes \{b, c_1, \dots, c_{n-2}\}.
\end{array}
\]
These would imply that  $u_{3,i}=0$ for all $2\leq i\leq n$ (here we use the homological stability
$H_{n-i}(\GL_{n-2}(F),\kk) \simeq H_{n-i}(\GL_{n-1}(F),\kk)$). Moreover, since  the map ${\alpha_1}_\ast|_{U_7}\colon U_7\arr T_4$ 
is induced by the homological stability  $H_{n-i-1}(\GL_{n-2}(F),\z)\simeq H_{n-i-1}(\GL_{n-1}(F),\z)$ for $1\leq i\leq n-2$,  
we have $u_7=0$.   Thus we may assume that 
\[
u=(u_{3,1}, u_4'') \in U_{3,1}\oplus U_4''\se H_n(\fff^2\times \GL_{n-2}(F),\kk).
\]

Let $u_{3,1}=\sum s\otimes z$, $u_4''=\sum a\otimes b\otimes \{c_1,\dots, c_{n-2}\}$. Then
\[
\alpha(u_{3,1}, u_4'')=(t_1', (t_2', t_2''), t_3', t_4')=(t_1,0,0,0),
\]
where
\[
\begin{array}{l}
t_1' =t_1=-\sum s\cup z,\\
t_2' =0=\sum  s\otimes z +\frac{(-1)^{n-2}}{(n-2)!}\sum \Big(b\otimes {\bf c}(aI_{n-2}, C_{1,n-2},\dots,C_{n-2,n-2})\\
\hspace{4.9cm}+a\otimes {\bf c}(bI_{n-2}, C_{1,n-2},\dots, C_{n-2,n-2})\Big),\\
t_2''=0=-\sum  \Big(b\otimes \{a, c_1,\dots, c_{n-2}\}+a\otimes \{b, c_1,\dots, c_{n-2}\}\Big),\\
t_3'=0=t_4'=0.
\end{array}
\]
 Therefore 
\[
\begin{array}{c}
t_1\!=\frac{(-1)^{n-2}}{(n-2)!}\!\sum\!\Big( {\bf c}(\diag(I_{n-2}, b), \diag(aI_{n-2},1), C_{1, n-1}, \dots ,C_{n-2, n-1})\\
\ \ \ \ \ \ \ \ \ \ \ \ \ \ \ \ \ \ \ \ + {\bf c}(\diag(I_{n-2},a), \diag(bI_{n-2},1),  C_{1, n-1}, \dots ,C_{n-2, n-1})\Big).
\end{array}
\]
This shows that $\kappa_n$ is surjective. 
\end{proof}

Based on the above theorem we make the following conjectures.

\begin{conj}\label{conj:exact}
{\rm (i)} For any $n$, the natural map
\[
\begin{array}{c}
\kappa_n\colon\!\B_n(F)\nn \!\arr\! \ker\Big(H_n(\GL_{n-1}(F),\zn)\! \arr\!
H_n(\GL_{n}(F),\zn)\Big)
\end{array}
\]
is surjective.
\par {\rm (ii)} For any $n\geq 3$, the sequence
\[
\begin{array}{c}
H_n(\fff^2\!\!\times\! \GL_{n-2}(F),\zn)\!\!\overset{{\alpha_1}_\ast-{\alpha_2}_\ast}{-\!\!\!-\!\!\!\larr} \!\!
H_n(\fff\!\!\times\! \GL_{n-1}(F),\zn)\\
\! \overset{\inc_\ast}{\larr}\! H_n(\GL_{n}(F),\zn) \!\arr \!  0
\end{array}
\]
is exact.
\end{conj}

\begin{cor}
If Conjecture $\ref{conj:exact}{\rm (ii)} $ holds for all $n\geq 3$, then $\kappa_n$ is surjective for all $n\geq 1$.
\end{cor}
\begin{proof}
The claim is trivial for $n=1,2$. So let $n\geq 3$. Let, by induction, the claim  holds for any $m <n$. Then $\kappa_m$ 
is surjective and its image is $(m-1)$-torsion. This implies that the map
\[
\begin{array}{c}
H_m(\GL_{m-1}(F),\zmm) \arr H_m(\GL_{m}(F),\zmm)
\end{array}
\]
is injective for any $m< n$. Now the claim follows from the previous theorem.
\end{proof}

\begin{rem}\label{rem:n-2}
(i) Conjecture \ref{conj:exact}(ii) holds for $n=3,4$ \cite[Corollary 3.5]{mirzaii-2008}, \cite[Theorem 2]{mirzaii2008}.
\par (ii) Conjecture \ref{conj:exact}(ii)  follows from Conjecture \ref{conj1} below, with $\zn$-coefficients. In fact by 
Conjecture \ref{conj1}, the groups $E_{0,n}^2(n,\zn)$ and $E_{1, n}^2(n,\zn)$ are trivial from which the exactness of 
the desired sequence follows.
\par (iii) Galatius, Kupers and Randal-Williams showed that Conjecture \ref{conj:exact} holds if we replace
$\zn$ with a field $\kk$ such that $(n-1)!\in \kk^\times$ (Theorem  \ref{G-K-R}).
\end{rem}

\begin{cor}\label{cor:n=4}
For any $n\leq 4$, $\kappa_n$ is surjective. In particular the natural maps
$H_3(\GL_2(F),\z\half) \arr H_3(\GL_3(F),\z\half)$ and $H_4(\GL_3(F),\z\Big[\frac{1}{6}\Big])\arr
H_4(\GL_4(F),\z\Big[\frac{1}{6}\Big])$ are injective.
\end{cor}
\begin{proof}
Clearly $\kappa_1$ and $\kappa_2$ are surjective. 
Conjecture \ref{conj:exact}(ii)  holds for $n=3$ by \cite[Corollary~3.5 (ii)]{mirzaii-2008} and for $n=4$ by \cite[Theorem~2]{mirzaii2008}.
Now the claim follows from Theorem \ref{thm:phi}.
\end{proof}

\begin{rem}\label{rem:k3}
(i) The case $n=3$ of the above corollary already was known. In fact
in \cite[Remark~3.5]{mirzaii2012}, $\kappa_3$ is defined and  is shown to be surjective. 
\par (ii) It is an open problem whether the map 
\[
H_3(\GL_2(F),\z)\arr H_3(\GL_3(F),\z)
\]
is injective (see \cite[Theorem~4.4]{mirzaii2015}). 
Generalising this, one might ask if, in fact, the map 
\[
\begin{array}{c}
H_n(\GL_{n-1}(F),\zn)\arr H_n(\GL_n(F),\zn)
\end{array}
\]
is injective?  Up to Conjecture \ref{conj:exact}(i) this is equivalent to the triviality of $\kappa_n$.
The answer to this question is positive when $\fff$ is divisible  (Theorem \ref{thm:div}) or when 
$F$ is real closed (Theorem \ref{real}).
\end{rem}

\section{{\bf The cokernel of \texorpdfstring{$H_{n+1}(\GL_{n-1}(F),\z) \arr H_{n+1}(\GL_{n}(F),\z)$}{Lg}}}

For any positive integers $n$ and $r$, let 
\[
H_\GL^{r}(F, n):=H_n(\GL_r(F),\z)/H_n(\GL_{r-1}(F),\z).
\]
By Theorem \ref{suslin},  $H^{r}_\GL(F, n)=0$ for $r>n$ and $H_\GL^{n}(F, n)\simeq K_n^M(F)$. By Theorem~\ref{G-K-R},
\[
H_\GL^{n}(F, n+1)_\kk\simeq H_{n+1}(\GL_{n}(F),\GL_{n-1}(F),\kk),
\]
where $\kk$ is a field such that $(n-1)!\in \kk^\times$. 

From the inclusions of groups $\GL_{n-1}(F)\harr \fff\times \GL_{n-1}(F) \harr \GL_n(F)$
we obtain the exact sequence
\[
\frac{H_{n+1}(\fff\!\!\times\! \GL_{n-1}(F),\!\z)}{H_{n+1}(\GL_{n-1}(F),\!\z)} 
\!\overset{\Phi}{\larr}\! 
H_\GL^{n}(F, n+1) \!\arr\! \frac{H_{n+1}(\GL_n(F),\z)}{H_{n+1}(\fff\!\!\times\! \GL_{n-1}(F),\!\z)}\!\arr\! 0.
\]
Let $H_{n+1}(\fff\times \GL_{n-1}(F),\z)=\bigoplus_{i=1}^6 S_i$, where
\[
\begin{array}{l}
S_1 =H_{n+1}(\GL_{n-1}(F),\z),\\
S_2  =\fff \otimes H_{n}(\GL_{n-1}(F),\z),\\
S_3=H_2(\fff,\z) \otimes H_{n-1}(\GL_{n-1}(F),\z) \\
S_4 = \bigoplus_{i=3}^{n+1} H_i(\fff,\z) \otimes H_{n-i+1}(\GL_{n-1}(F),\z),\\
S_5=\tors(\fff, H_{n-1}(\GL_{n-1}(F),\z)),\\
S_6=\bigoplus_{i=2}^n\tors(H_i(\fff,\z), H_{n-i}(\GL_{n-1}(F),\z)).\\
\end{array}
\]
Clearly $\Phi(S_1)=0$.
By homological stability 
\[
H_{n-i+1}(\GL_{n-2}(F),\z)\simeq H_{n-i+1}(\GL_{n-1}(F),\z)
\]
for $3\leq i\leq n+1$  and thus  $\Phi(S_4)=0$.  Furthermore, for any $2\leq i\leq n$, 
the homological stability  $H_{n-i}(\GL_{n-2}(F),\z)\simeq H_{n-i}(\GL_{n-1}(F),\z)$ induces
the isomorphism
\[
\tors(H_i(\fff\!,\!\z),\! H_{n-i}(\GL_{n-2}(F),\z))\!\simeq\! \tors(H_i(\fff\!,\!\z),\! H_{n-i}(\GL_{n-1}(F),\z)).
\]
Thus we may assume that $\Phi(S_6)=0$. Moreover,
\[
\begin{array}{l}
\Phi \Big(\im(\fff \otimes H_{n}(\GL_{n-2}(F),\z)\overset{\id_\fff\otimes \inc_\ast}{-\!\!\!-\!\!\!-\!\!\!-\!\!\!-\!\!\!-\!\!\!-\!\!\!-\!\!\!\larr} S_2)\Big)=0,\\
\Phi \Big(\im(H_2(\fff,\z) \otimes H_{n-1}(\GL_{n-2}(F),\z)
\overset{\id_{H_2(\fff,\z)}\otimes \inc_\ast}{-\!\!\!-\!\!\!-\!\!\!-\!\!\!-\!\!\!-\!\!\!-\!\!\!-\!\!\!-\!\!\!-\!\!\!\larr} S_3)\Big)=0.
\end{array}
\]
Thus the above exact sequence finds the following form
\begin{equation*}
\fff \otimes H_\GL^{n-1}(F, n) \oplus \bigwedge{}_\z^2\fff\otimes K_{n-1}^M(F)\oplus \tors(\fff, H_{n-1}(\GL_{n-1}(F),\z))  
\end{equation*}
\[
\!\overset{\Phi}{\larr}\! H_\GL^{n}(F, n+1) \larr \frac{H_{n+1}(\GL_n(F),\z)}{H_{n+1}(\fff\times \GL_{n-1}(F),\z)}\arr 0,
\]
where $K_{n-1}^M(F)\!\simeq\! H_{n-1}(\GL_{n-1}(F),\z)/H_{n-1}(\GL_{n-2}(F),\z)$ (Theorem~\ref{suslin}).

Now in the above exact sequence, replace the coefficients $\z$ with $\kk:=\zn$. 
By Corollary \ref{cor:dec},  
\[
H_{n-1}(\GL_{n-1}(F),\kk)\simeq \im(H_{n-1}(\GL_{n-2}(F),\kk))\oplus K_{n-1}^M(F)_\kk.
\] 
Since $\Phi \Big(\im(\tors(\fff, H_{n-1}(\GL_{n-2}(F),\z))_\kk\overset{\tors(\id_\fff, \inc_\ast)}{-\!\!\!-\!\!\!-\!\!\!-\!\!\!-\!\!\!-\!\!\!-\!\!\!-\!\!\!-\!\!\!\larr} 
{(S_5)}_\kk)\Big)=0$, we obtain the exact sequence
\begin{equation*}
\fff \otimes H_\GL^{n-1}(F, n)_\kk
\oplus \bigwedge{}_\z^2\fff\otimes K_{n-1}^M(F)_\kk\oplus \tors(\fff, K_{n-1}^M(F))_\kk
\end{equation*}
\[
\!\overset{\Phi}{\larr}\! H_\GL^{n}(F, n+1)_\kk \larr \frac{H_{n+1}(\GL_n(F),\kk)}{H_{n+1}(\fff\times \GL_{n-1}(F),\kk)}\arr 0.
\]
This suggests that to study $H_\GL^{n}(F, n+1)\nn$, up to induction, we need to study the cokernel of
\[
\begin{array}{c}
H_{n+1}(\fff\times \GL_{n-1}(F),\zn) \arr H_{n+1}(\GL_n(F),\zn).
\end{array}
\]
We will do this in the next two sections.

For now, we look at  $H_\GL^{n}(F, n+1)$ for $n=1$ and $n=2$.  Clearly
\[
\begin{array}{l}
H_\GL^1(F, 2)= H_2(\GL_1(F),\z)/H_2(\GL_0(F),\z)\simeq \bigwedge_\z^2\fff.
\end{array}
\]
But $H_\GL^2(F, 3)$ has much richer structure.  

\begin{prp}\label{mirzaii}
For any infinite field $F$, we have the exact sequence
\[
H_3(\SL_2(F),\z)_\fff  \arr  H_\GL^2(F, 3) \arr \fff\otimes K_2^M(F)\arr K_3^M(F)/2 \arr 0.
\] 
Moreover, we have the decomposition
\begin{equation*}
\begin{array}{c}
H_\GL^2(F, 3) \half\simeq K_3^\ind(F)[\frac{1}{2}] \oplus \fff\otimes K_2^M(F)[\frac{1}{2}].
\end{array}
\end{equation*}
\end{prp}
\begin{proof}
By studying  the relative Lyndon/Hochschild-Serre spectral sequence \cite[Theorem 1.4]{essert2013} associated to the extension
\[
1 \arr \SL_2(F)  \arr  \GL_2(F)  \overset{\det}{\larr}   \fff \arr 1,
\]
relative to the subgroup $\GL_1(F) \se \GL_2(F)$, we get the exact sequence
\[
H_3(\SL_2(F),\z)_\fff  \arr  H_\GL^2(F, 3) \arr H_1(\fff, H_2(\SL_2(F),\z)) \arr 0.
\] 
The inclusion $\SL_2 \larr \SL_3$ induces the short exact sequence
\begin{gather*}\label{hutch-tao1}
0 \arr H_1(\fff, H_2(\SL_2(F),\z))\arr H_1(\fff, H_2(\SL_3(F),\z)) \arr K_3^M(F)/2 \arr 0
\end{gather*}
\cite[Theorem 3.2]{H-T2009}. Since $H_2(\SL_3(F),\z)\simeq K_2(F)$ \cite[2.1]{sah1989},
$\fff$ acts trivially on $H_2(\SL_3(F),\z)$. Thus
\[
H_1(\fff, H_2(\SL_3(F),\z))\simeq \fff \otimes K_2^M(F).
\]
Moreover, the map  $\fff \otimes K_2^M(F)\simeq H_1(\fff, H_2(\SL_3(F),\z)) \arr K_3^M(F)/2$ is induced by the natural map 
$\fff \otimes K_2^M(F) \arr K_3^M(F)$ (see the proof of \cite[Theorem 3.2]{H-T2009}). Thus we have the exact sequence
\[
H_3(\SL_2(F),\z)_\fff  \arr  H_\GL^2(F, 3) \arr \fff\otimes K_2^M(F)\arr K_3^M(F)/2 \arr 0.
\]
This proves the first claim.

By \cite[Theorem~6.1]{mirzaii-2008}, the map $H_{3}(\SL_{2}(F), \zzz)_\fff\arr H_\GL^2(F, 3) \half$
is injective. Moreover, $H_{3}(\SL_{2}(F), \z\half)_\fff \simeq K_3^\ind(F)\half$ by \cite[Proposition~6.4]{mirzaii-2008} or
\cite[Theorem~3.7]{mirzaii2012}. Thus we have the exact sequence
\begin{equation*}\label{H23}
\begin{array}{c}
0\arr K_3^\ind(F)[\frac{1}{2}]\arr H_\GL^2(F, 3) \half\arr \fff\otimes K_2^M(F)[\frac{1}{2}]\arr 0.
\end{array}
\end{equation*}

To finish the proof, we should show that the above short exact sequence  splits. Let $\phi$ be the composite
\[
\fff \otimes K_2^M(F) \overset{\id_\fff \otimes \iota_2}{-\!\!\!-\!\!\!-\!\!\!-\!\!\!\larr}
\fff \otimes H_2(\GL_2(F),\z) \overset{\cup}{\arr} H_3(\fff \times \GL_2(F),\z) 
 \]
 \[
 \overset{\tau_\ast}{\larr} H_3(\GL_2(F),\z)\arr H_\GL^2(F, 3),
\]
where $\tau\colon \fff \times \GL_2(F) \arr \GL_2(F)$ is given by $(a, A) \mt aA$. By \cite[Lemma~3.2 (ii)]{mirzaii2015}, 
the composite 
\[
\fff \otimes K_2^M(F) \overset{\phi}{\arr} H_\GL^2(F, 3)  \arr \fff \otimes K_2^M(F)
\]
coincides with multiplication by $2$. Now it is easy to construct a splitting map.
\end{proof}

In the next section we will describe $H_\GL^2(F, 3)$ in a different way (Proposition \ref{prp:n=2}).

\begin{rem}\label{sl2-k3}
For any field $F$ there is a natural map 
\[
H_3(\SL_2(F),\z)_\fff \arr K_3^\ind(F)
\]
(see \cite[\S 1]{mirzaii2015} for its construction). It is an open question, asked by Suslin, whether this map 
is an isomorphism \cite[Question 4.4]{sah1989}. Hutchinson and Tao have proved that it is surjective \cite[Lemma 5.1]{H-T2009}.
For more on this question see \cite[Theorem 4.4]{mirzaii2015}.
\end{rem}

\section{{\bf The cokernel of \texorpdfstring{$H_{n+1}(\fff\!\times \GL_{n-1}(F),\z) \arr H_{n+1}(\GL_{n}(F),\z)$}{Lg}}}\label{sp}

To study the quotient group $H_{n+1}(\GL_n(F),\z)/H_{n+1}(\fff\times \GL_{n-1}(F),\z)$, we look at a certain 
spectral sequence introduced and studied in \cite[\S 2.2]{elbaz1998}, \cite[Section 3]{mirzaii-2008} and \cite[Section 5]{mirzaii2008}.

Let $\kk$ be a commutative ring. For any $l\geq 0$,  let  $D_l(F^n)$ be the free $\kk$-module with a basis consisting of 
$(l+1)$-tuples $(\lan w_0\ran, \dots, \lan w_l\ran)$, where $0\neq w_i\in F^n$,  $\lan w_i\ran=Fw_i$ and $\lan w_i \ran \neq \lan w_j \ran$ 
when $i\neq j$. Set $D_{-1}(F^n):=\kk$. The group $\GL_n(F)$ acts naturally on  $D_l(F^n)$ on the left as follow:
\[
g.(\lan w_0\ran, \dots, \lan w_l\ran):=(\lan gw_0\ran, \dots, \lan gw_l\ran).
\]
We consider $D_{-1}(F^n)=\kk$ as  a trivial $\GL_n(F)$-module. If it is necessary we convert these left actions to right actions
by the definition $m.g:=g^{-1}.m$.

Let define $\partial_0\colon D_0(F^n) \arr D_{-1}(F^n)=\kk$ by $\sum_i n_i(\lan w_i\ran) \mt \sum_i n_i$.  For $l\geq 1$,
we define the $l$-th differential operator $\partial_l\colon D_l(F^n) \arr D_{l-1}(F^n)$, 
as an alternating sum of face operators which throw away the $i$-th component of generators.

For any integer $l\geq 0$, set $M_l=D_{l-1}(F^n)$. It is easy to see that the complex of $\GL_n(F)$-modules
\[
M_\bullet: \ \ \ \  0 \leftarrow M_0 \leftarrow M_1 \leftarrow \cdots \leftarrow M_l \leftarrow  \cdots
\]
is exact (see the proof of \cite[Lemma~2.2]{suslin1985}). 

Take a projective resolution $P_\bullet \arr  \z$ of $\z$ over $\GL_n(F)$. From the double complex 
$M_\bullet \otimes _{\z[\GL_n(F)]} P_\bullet$ 
we obtain the first quadrant spectral  sequence converging  to zero with $E_{\bullet,\bullet}^1$-terms
\begin{gather*}
{\eee}_{p, q}^1(n,\kk)=\begin{cases}
H_q(\fff^{p} \times \GL_{n-p}(F),\kk)& \text{if $0 \le p \le 2$}\\
 H_q(\GL_n(F), M_{p})& \text{if $p \ge 3,$} \end{cases}
\end{gather*}
(see \cite[Section~3]{mirzaii-2008}, \cite[Section~5]{mirzaii2008}). It is easy to see that $d_{1, q}^1(n,\kk)=\inc_\ast$. In particular
\[
\eee_{0,q}^2(n,\kk)=\frac{H_q(\GL_n(F),\kk)}{H_q(\fff\times \GL_{n-1}(F),\kk)}\cdot
\]
Moreover, $d_{2, q}^1(n,\kk)={\alpha_1}_\ast-{\alpha_2}_\ast$ which is discussed  in the exact sequence (\ref{exactness}). Now we 
would like to describe $\eee_{3, q}^1(n, \kk)$. The orbits of the action of  $\GL_n(F)$ on $M_3=D_2(F^n)$ are represented by
\[
w_1=(\lan e_1 \ran, \lan e_2 \ran, \lan e_3 \ran) \ \ \text{and} \ \ w_2=(\lan e_1 \ran, \lan e_2 \ran, \lan e_1+ e_2\ran).
\]
Thus
\begin{align*}
\eee_{3, q}^1(n, \kk) & \simeq H_q(\stabe_{\GL_n(F)}(w_1),\kk) \oplus H_q(\stabe_{\GL_n(F)}(w_2),\kk)\\
                               &\simeq H_q(\fff^3\times \GL_{n-3}(F),\kk) \oplus H_q(\fff I_2 \times \GL_{n-2}(F),\kk),
\end{align*}
where $\fff I_2:=\{aI_2: a\in \fff\}$ (see \cite[Theorem 1.11]{N-S1990}). Moreover, 
\begin{align*}
d_{3, q}^1(n,\kk)|_{H_q(\fff^3\times \GL_{n-3}(F),\kk)}&={\sigma_1}_\ast-{\sigma_2}_\ast+{\sigma_3}_\ast,\\
d_{3, q}^1(n,\kk)|_{H_q(\fff I_2 \times \GL_{n-2}(F),\kk)}& =\inc_\ast.
\end{align*}
Note that $\diag(a,b,c, A)\overset{\sigma_1}{\mt} \diag(b,c, a, A)$,  $\diag(a,b,c, A)\overset{\sigma_2}{\mt} \diag(a,c, b, A)$ 
and $\diag(a,b,c, A)\overset{\sigma_3=\inc}{\mt}\diag(a,b,c, A)$.

When $\kk$ is a field, with a similar method as in \cite[Lemma~4.2]{mirzaii2005}, one can show that $\eee_{p,q}^2(n,\!\kk)\!=\!0$ for 
$p=0,1$ and $q\leq n-1$. Moreover, by Theorem~\ref{G-K-R}, this is also true for $q=n$ if $(n-1)!\in \kk^\times$.  

\begin{conj}\label{conj1}
Let $n\geq 3$ and let $\kk$ be either a field or a subring of $\q$.  Then for any $3\leq i\leq n+2$, $\eee_{i,n-i+2}^2(n,\kk)=0$. 
\end{conj}

\begin{rem}
(i) Since the spectral sequence converges to zero, it follows from Conjecture \ref{conj1} that $\eee_{1,n}^2(n,\kk)=0$.
For $\kk=\zn$ this is equivalent to Conjecture \ref{conj:exact}(ii). Thus Conjecture \ref{conj1} implies Conjecture~\ref{conj:exact}.
\par (ii) Moreover, it follows from Conjecture \ref{conj1} that the differential 
\[
d_{2,n}^2(n,\kk)\colon \eee_{2,n}^2(n,\kk)\arr \eee_{0,n+1}^2(n,\kk)
\]
is surjective. For $\kk=\zn$, this differential is used in the construction of the map $\chi_{n}$ of Theorem B.
In fact, as will be clear from the definition of $\chi_n$, Conjecture \ref{conj1} implies Conjecture \ref{conj:chi} below.
\par (iii)  Conjecture  \ref{conj1} is known for $n=3, 4$ \cite[\S 3]{mirzaii-2008}, \cite[\S 6]{mirzaii2008}. 
\end{rem}

Now we study the group  $H_{n+1}(\GL_n(F),\z)/H_{n+1}(\fff\times \GL_{n-1}(F),\z)$ for $n=1$ and $n=2$.
Clearly $H_{2}(\GL_1(F),\z)/H_{2}(\fff\times \GL_{0}(F),\z)$ is trivial. The group $H_{3}(\GL_2(F),\z)/H_{3}(\fff\times \GL_{1}(F),\z)$ 
is more interesting and is connected to the Bloch group of  $F$.

\begin{prp}\label{prp:n=2}
For any infinite field $F$, 
\[
H_{3}(\GL_2(F),\z)/H_{3}(\fff\times \GL_{1}(F),\z)\simeq \B_2(F).
\]
In particular, we have the exact sequence
\[
\begin{array}{c}
\tors(\mu(F),\mu(F))\oplus \fff \otimes H_\GL^1(F,2) \arr H_\GL^2(F,3)\arr \B_2(F)\arr 0
\end{array}
\]
where $\mu(F)$ is the group of roots of unity in $F$.
\end{prp}
\begin{proof}
By studying the spectral sequence $E_{\bullet,\bullet}^1(2,\z)$, 
Suslin showed that there is a natural map $H_3(\GL_2(F), \z) \arr B(F)$ such that the sequence
\begin{equation}\label{sus-exact}
H_3(\GM_2(F),\z)\arr H_3(\GL_2(F), \z) \arr B(F)\arr 0
\end{equation}
is exact, where $\GM_2(F)$ is the subgroup of monomial matrices in $\GL_2(F)$ \cite[Theorem 2.1]{suslin1991}.
Note that  $\GM_2(F)\simeq (\fff \times \fff)\rtimes \Sigma_2=T_2(F) \rtimes \Sigma_2$,
where $\Sigma_2=\Bigg\{1:=I_2=
{\mtx 1 0 0 1}, \sigma:={\mtx 0 1 1 0} \Bigg\}$.

The Lyndon/Hochschild-Serre spectral sequence associated to the
extension $1 \arr T_2(F) \arr \GM_2(F) \arr \Sigma_2\arr 1$, i.e.
\[
\EE_{p,q}^2=H_p(\Sigma_2, H_q(T_2(F),\z))\Rightarrow H_{p+q}(\GM_2(F),\z),
\]
gives us a filtration of $H_3(\GM_2(F),\z)$,
\[
0=F_{-1}H_3(\GM_2(F),\z) \se \cdots \se F_3H_3(\GM_2(F),\z)=H_3(\GM_2(F),\z),
\]
such that
\[
\begin{array}{l}
\vspace{1.5 mm}
\EE^\infty_{0, 3}\simeq F_0H_3(\GM_2(F),\z)=H_3(T_2(F),\z)_{\Sigma_2},\\
\vspace{1.5 mm}
\EE^\infty_{1, 2} \simeq F_1H_3(\GM_2(F),\z)/F_0H_3(\GM_2(F),\z)\simeq \EE_{1,2}^2,\\
\vspace{1.5 mm}
\EE^\infty_{2, 1}\simeq F_2H_3(\GM_2(F),\z)/F_1H_3(\GM_2(F),\z)=0,\\
\EE^\infty_{3, 0}\simeq 
H_3(\GM_2(F),\z)/F_2H_3(\GM_2(F),\z)\simeq H_3(\Sigma_2,\z),
\end{array}
\]
(see \cite[Section~4]{M-M2015} for some details).
Thus
\begin{align*}
\EE_{1, 2}^2 \simeq F_2H_3(\GM_2(F),\z)/H_3(T_2(F),\z).
\end{align*}
Moreover, since the map $\GM_2(F) \arr \Sigma_2$ splits, we obtain the decomposition 
\[
H_3(\GM_2(F),\z)\simeq F_2H_3(\GM_2(F),\z)\oplus H_3(\Sigma_2,\z).
\]
The matrix $\sigma={\mtx 0 1 1 0}$ is conjugate in $\GL_2(F)$ to
${\mtx 1 1 0 {-1}} \in U_2(F)$, where
\[
U_2(F)=\Bigg\{{\mtx a b 0 d }: a, d\in \fff, b\in F\Bigg\}.
\]
Hence $\im (H_3(\Sigma_2,\z)) \se \im (H_3(U_2(F),\z))=\im (H_3(T_2(F),\z))$. Thus from the exact sequence
(\ref{sus-exact}) we obtain the exact sequence
\begin{equation}\label{exact-gl2}
\EE_{1, 2}^2 \arr {H_{3}(\GL_2(F),\z)}/{H_{3}(\fff\times \GL_{1}(F),\z)}\arr B(F)\arr 0.
\end{equation}
Observe that
\begin{align*}
\EE_{1, 2}^2& \simeq H_1(\Sigma_2, \fff\otimes \fff)\simeq  H_1(\Sigma_2, \mu(F)\otimes \mu(F))\\
& \simeq  H_1(\Sigma_2, \mu_{2^\infty}(F)\otimes \mu_{2^\infty}(F))\simeq  
\frac{(\mu_{2^\infty}(F)\otimes \mu_{2^\infty}(F))^{\Sigma_2}}{(1+\sigma)( \mu_{2^\infty}(F)\otimes \mu_{2^\infty}(F))}\\
& =( \mu_{2^\infty}(F)\otimes \mu_{2^\infty}(F))^{\Sigma_2}={}_2(\mu_{2^\infty}(F)\otimes_\z\mu_{2^\infty}(F))\\
& \simeq
\begin{cases}
0 & \text{if $\mu_{2^\infty}(F)$ is infinite or $\char(F)=\! 2$} \\
\z/2 &   \text{if $\mu_{2^\infty}(F)$ is finite and $\char(F)\neq 2$,}
\end{cases}
\end{align*}
(see \cite[Lemma 4.1 and page 5088]{M-M2015}). Recall that $\mu_{2^\infty}(F)$ is the group of
$2$-power roots of unity in $F$, i.e. $\mu_{2^\infty}(F)=\{a\in F: a^{2^n}=1\ \text{for some $n\geq 0$}\}$ and 
for an abelian group $A$, ${}_2A:=\{a\in A: 2a=0\}$.

Suslin also showed that there is a natural map $H_3(\GL_3(F), \z) \arr B(F)$ such that the sequence
\begin{equation}\label{sus-exact2}
H_3(\GM_2(F),\z)\oplus H_3(T_3(F),\z) \arr H_3(\GL_3(F), \z) \arr B(F)\arr 0
\end{equation}
is exact \cite[Proposition 3.1]{suslin1991} and the map $H_3(\GL_2(F), \z) \arr B(F)$ in (\ref{sus-exact})
factors through $H_3(\GL_3(F), \z)$. Now from this one can obtain the exact sequence
\[
0\arr T_F \arr H_3(\GL_3(F), \z)/L \arr B(F)\arr 0,
\]
where $L=\fff^{\otimes 3}\oplus \fff\otimes H_2(\fff, \z)\oplus H_3(\fff, \z)$ and 
$T_F$ sits in the following exact sequence
\[
0\arr \tors (\mu(F),\mu(F)) \arr T_F \arr H_1(\Sigma_2, \mu_{2^\infty}(F)\otimes \mu_{2^\infty}(F)) \arr 0.
\]
Moreover,
\[
H_3(\GL_3(F), \z)/L\simeq K_3^\ind(F)
\]
(see the proof \cite[Theorem 5.1]{M-M2015}). From these results we obtain the exact sequence
\[
0\arr H_1(\Sigma_2, \mu_{2^\infty}(F)\otimes \mu_{2^\infty}(F)) \arr K_3^\ind(F)/\tors (\mu(F),\mu(F))\arr B(F) \arr 0.
\]
Observe that $\EE_{1,2}^2\simeq H_1(\Sigma_2, \mu_{2^\infty}(F)\otimes \mu_{2^\infty}(F))$.
Now from the commutative diagram with exact rows
\[
\begin{tikzcd}
& \EE_{1, 2}^2 \ar[r]\ar[d, "="] & \displaystyle\frac{H_{3}(\GL_2(F),\z)}{H_{3}(\fff\times \GL_{1}(F),\z)}\ar[r]\ar[d] & B(F)\ar[r]\ar[d, "="] & 0\\
0 \ar[r]& \EE_{1, 2}^2 \ar[r]          &  K_3^\ind(F)/\tors (\mu(F),\mu(F))\ar[r]                                      & B(F) \ar[r]      & 0
\end{tikzcd}
\]
 it follows that 
\[
\frac{H_{3}(\GL_2(F),\z)}{H_{3}(\fff\times \GL_{1}(F),\z)}\simeq K_3^\ind(F)/\tors (\mu(F),\mu(F))\simeq \B_2(F).
\]
This proves the first claim.

To prove the second claim, note that the image of the map 
\[
\fff\otimes H_{2}(\GL_{1}(F),\z)\overset{\cup}{\larr} H_\GL^2(F,3)
\]
 is generated by the elements
${\bf c}(\diag(a,1), \diag(1,b), \diag(1,c))$, $a,b,c\in \fff$, and the image of 
\[
\begin{array}{c}
\bigwedge_\z^2\fff\otimes H_{1}(\GL_{1}(F),\z)\overset{\cup}{\larr}H_\GL^2(F,3)
\end{array}
\]
is generated by the elements
${\bf c}(\diag(d,1), \diag(e,1), \diag(1,f))$, $d,e,f\in \fff$. The conjugation by $\sigma={\mtx 0 1 1 0}$ on $\GL_2(F)$, i.e.
the map $i_\sigma: \GL_2(F)\arr \GL_2(F)$, $A\mt \sigma A\sigma^{-1}$, induces identity on $H_3(\GL_2(F), \z)$. 
This implies that
\begin{align*}
{\bf c}(\diag(a,1), \diag(1,b), \diag(1,c))& ={\bf c}(\diag(1,a), \diag(b,1), \diag(c,1))\\
& ={\bf c}(\diag(b,1), \diag(c,1), \diag(1,a)).
\end{align*}
Thus the images of $\fff\otimes H_{2}(\GL_{1}(F),\z)$ and 
$\bigwedge_\z^2\fff\otimes H_{1}(\GL_{1}(F),\z)$ coincide in $H_\GL^2(F,3)$. 
\end{proof}

\begin{rem}\label{rem:H3}
(i) Proposition \ref{prp:n=2} holds for any field with more than nine elements or more generally for any local domain where its
residue field has more than nine elements. The proof is the same, but for the exact sequences (\ref{sus-exact}) and (\ref{sus-exact2}) 
see \cite[Section 5]{mirzaii2017}.
\par (ii) Putting the exact sequences of Propositions \ref{mirzaii} and \ref{prp:n=2} in one diagram we get the commutative diagram
\[
\begin{tikzcd}[cramped]
\tors(\mu(F),\mu(F))\ar[d, "{(\id, 0)}"]\ar[r]  & H_3(\SL_2(F),\z)_\fff\ar[d]\ar[dr, two heads] &   \\
\tors(\mu(F),\mu(F))\!\oplus\!\fff\!\otimes H_{\GL}^1\!(F,2)\! \ar[r]\ar[dr, two heads, "\theta"] &\! H_{\GL}^2(F,3)\! \ar[r]\ar[d] &\! \B_2(F)\!\arr 0,\\
   & \fff \!\otimes K_2^M(F) \ar[d] &   \\
    & K_3^M(F)/2 \ar[d] &  \\ 
        & 0 &  
\end{tikzcd}
\]
where $\theta|_{\fff\otimes H_{\GL}^1(F,2)}(c\otimes(a\wedge b))=-a\otimes \{b,c\}+b\otimes\{a,c\}$ \cite[Lemma~3.2]{mirzaii2015},
$\theta|_{\tors(\mu(F),\mu(F))}=0$ and $\tors (\mu(F),\mu(F)) \arr H_3(\SL_2(F),\z)_\fff$  is induced by the map $\mu(F) \arr \SL_2(F)$
given by $a \mt \diag(a, a^{-1})$. Furthermore, the surjective map $H_3(\SL_2(F),\z)_\fff\arr \B(F)$ factors through $K_3^\ind(F)$.
\end{rem}

\section{{\bf The map \texorpdfstring{$\chi_n$}{Lg}}}\label{chi}

The following is Theorem B of the introduction.

\begin{thm}\label{thm:surj0}
Let  the natural map 
\[
\begin{array}{c}
H_m(\GL_{m-1}(F), \zmm) \arr H_m(\GL_{m}(F), \zmm)
\end{array}
\]
be injective for $m=n-1, n-2$. Then there is a natural map
\[
\begin{array}{c}
\chi_n\colon \B_n(F)\nn \larr \displaystyle\frac{H_{n+1}( \GL_{n}(F),\z\nn)}{H_{n+1}(\fff\times \GL_{n-1}(F),\z\nn)}\cdot
\end{array}
\]
If Conjecture  $\ref{conj1}$ holds with $\zn$ coefficients, then $\chi_n$ is surjective. 
\end{thm}
\begin{proof}
It is enough to construct $\chi_n$. The second part follows from the construction of $\chi_n$ below and an easy analysis of the spectral
sequence $E_{p,q}^1(n, \z\nn)$.

The map $\chi_1$ is the trivial map  and $\chi_2$ is the isomorphism 
\[
\B_2(F)\simeq H_3(\GL_2(F),\z)/H_3(\fff\times \GL_1(F),\z)
\] 
of Proposition~\ref{prp:n=2}. So let $n\geq 3$.

For simplicity, set $\kk=\zn$. We construct a surjective map 
\[
\chi_n'\colon \B_n(F)_\kk\two \eee_{2,n}^2(n,\kk). 
\]
Then $\chi_n$ is  defined as the composite of $\chi_n'$ with the differential  $d_{2, n}^2(n,\kk)$:
\[
\chi_n:=d_{2, n}^2(n,\kk)\circ\chi_n'.
\]

The group $\eee_{2,n}^2(n,\kk)$ is the homology of the complex
\[
\hspace{-2.5cm}
H_n(\fff\! I_2\times \GL_{n-2}(F),\kk)\oplus H_n(\fff^3\times \GL_{n-3}(F),\kk)  \overset{d_{3,n}^1(n,\kk)}{-\!\!\!-\!\!\!-\!\!\!\larr}
\]
\[
\hspace{2.5cm}
H_n(\fff^2\times \GL_{n-2}(F),\kk)  \overset{d_{2,n}^1(n,\kk)}{-\!\!\!-\!\!\!-\!\!\!\larr} H_n(\fff \times \GL_{n-1}(F),\kk).
\]
(Observe that this complex already appeared in the proof of Theorem \ref{thm:phi}(ii). In fact $d_{3,n}^1(n,\kk)=\beta$ 
and $d_{2,n}^1(n,\kk)=\alpha$.) Consider the decompositions
\[
\begin{array}{l}
H_n(\fff \times \GL_{n-1}(F),\kk) =\bigoplus_{i=1}^4T_i, \\
H_n(\fff^2 \times \GL_{n-2}(F),\kk) =\bigoplus_{i=1}^{11} U_i, 
\end{array}
\]
as in the proof of Theorem \ref{thm:phi}(ii)
with $T_2=T_2'\oplus T_2''$,  $U_4=U_4'\oplus U_4''$ and $U_3=\bigoplus_{i=1}^n U_{3,i}$.

Let $u=(u_1,\dots,u_{11})$ be in the kernel of $d_{2,n}^1(n,\kk)={\alpha_{1}}_\ast-{\alpha_{2}}_\ast$, where $u_4=(u_4', u_4'')$
and $u_3=(u_{3,i})_{1\leq i\leq n}$. As  in the proof of Theorem \ref{thm:phi}(ii), we can show that we may assume that
$u_1,u_2, (u_{3,i})_{2\leq i\leq n}, u_4', u_5, u_6, u_7,u_8, u_9,u_{10}$ and $u_{11}$ are trivial. Hence
\[
u=(u_{3,1}, u_4'')\in U_{3,1}\oplus U_4''\se H_n(\fff^2 \times \GL_{n-2}(F),\kk).
\]
If $u_{3,1}=\sum s\otimes z$ and $u_4''= \sum a\otimes b\otimes \{c_1,\dots, c_{n-2}\}$, then as in 
the proof of Theorem \ref{thm:phi}(ii),
\[
\begin{array}{c}
d_{2,n}^1(n,\kk)(u)\!=\!(t_1',(t_2',t_2''), t_3', t_4')\!=\!0\in H_n(\fff\! \times\! \GL_{n-1}(F),\kk)=\bigoplus_{i=1}^4\! T_i,
\end{array}
\] 
where
\[
\begin{array}{rl}
t_1'&\!\!\!\!=-\sum s\cup z=0,\\
t_2' &\!\!\!\!=0=\sum s\otimes z+ \frac{(-1)^{n-2}}{(n-2)!}\sum\Big(b \otimes{\bf c}(aI_{n-2}, C_{1, n-2}, \dots,C_{n-2, n-2})\\
&\hspace{4.7cm}+a \otimes{\bf c}(bI_{n-2}, C_{1, n-2}, \dots, C_{n-2, n-2})\Big),\\
t_2'' &\!\!\!\!=0=-\sum \Big(b \otimes \{a, c_1, \dots, c_{n-2}\} +a \otimes \{b, c_1, \dots, c_{n-2}\}\Big),\\
t_3'&\!\!\!\! =0=t_4'.
\end{array}
\]
These calculations show that the map 
\[
\chi_n'\colon \ker(\delta_2^{(n)})_\kk \larr  \eee_{3,n}^2(n,\kk),
\]
given by
\[
\begin{array}{c}
\sum a\otimes b\otimes \{c_1,\dots, c_{n-2}\} \mt (u_{3,1}, \sum a\otimes b\otimes \{c_1,\dots, c_{n-2}\})
\end{array}
\]
is surjective, where
\[
\begin{array}{rl}
u_{3,1} &\!\!\!\!=-\frac{(-1)^{n-2}}{(n-2)!}\sum\Big(b \otimes{\bf c}(aI_{n-2}, C_{1, n-2}, \dots ,C_{n-2, n-2})\\
& \hspace{2.1 cm}+a \otimes{\bf c}(bI_{n-2}, C_{1, n-2}, \dots, C_{n-2, n-2})\Big).
\end{array}
\]
To finish the proof, we should check that this map  factors through $\B_n(F)_\kk$ and for this we should show that
\[
\chi_n'\Big(b\otimes c\otimes  \{a, d_1, \dots, d_{n-3}\}+a\otimes c \otimes \{b, d_1, \dots, d_{n-3}\}+
a\otimes b\otimes \{c, d_1, \dots, d_{n-3}\}\Big)
\]
is trivial. Consider the following summands of $H_n(\fff^3\times \GL_{n-3}(F),\kk)$:
\[
\begin{array}{l}
W''' =F_1^\times\otimes F_2^\times \otimes F_3^\times \otimes H_{n-3}(\GL_{n-3}(F),\kk)\\
\ \ \ \ \ \!\ =W_1'''\oplus W_2''' ,\\
W_1''' =F_1^\times\otimes F_2^\times\otimes F_3^\times \otimes \im(H_{n-3}(\GL_{n-4}(F),\kk)),\\ 
W_2'''=F_1^\times\otimes F_2^\times\otimes  F_3^\times \otimes K_{n-3}^M(F)_\kk.
\end{array}
\]
Then $d_{3,n}^1(n,\kk)|_{W_2'''}$ factors through $U_4$, and 
\[
\begin{array}{c}
d_{3,n}^1(n,\kk)|_{W_2'''}\colon W_2''' \arr U_4=U_4'\oplus U_4'' \se H_n(\fff^2\times \GL_{n-2}(F),\kk),\\
a\otimes b\otimes c \otimes \{d_1, \dots, d_{n-3}\} \mt (u', u''),
\end{array}
\]
where
\[
\begin{array}{rl}
u'&\!\!\!\! =-\!\frac{(-1)^{n-3}}{(n-3)!}\Big(b \otimes c\otimes {\bf c}(aI_{n-3}, D_{1, n-3}, \dots ,D_{n-3, n-3}) \\
&\hspace{1.2cm}+a \otimes c \otimes {\bf c}(bI_{n-3}, D_{1, n-3}, \dots, D_{n-3, n-3})\\
& \hspace{1.2cm} +a \otimes b \otimes {\bf c}\Big(cI_{n-3}, D_{1, n-3}, \dots, D_{n-3, n-3})\Big),\\
u'' &\!\!\!\!= b\otimes c\otimes  \{a, d_1, \dots, d_{n-3}\}+a\otimes c \otimes \{b, d_1, \dots, d_{n-3}\}\\
&\ +a\otimes b\otimes \{c, d_1, \dots, d_{n-3}\}.
\end{array}
\]

On the other hand, for $W'=F_2^\times\otimes F_3^\times\otimes H_{n-2}(\GL_{n-3}(F),\kk)$ the summand
of $H_n(\fff^3\times \GL_{n-3}(F),\kk)$, we have
\[
\begin{array}{c}
d_{3,n}^1(n,\kk)|_{W'\oplus W_2'''}\colon\! W'\!\oplus\! W_2''' \arr U_{3,1}\! \oplus\! U_4'\oplus\! U_4'' \se H_n(\fff^2\!\times\! \GL_{n-2}(F),\kk), \\
 (-u', a\otimes b\otimes c \otimes \{d_1, \dots, d_{n-3}\} ) \mt(\frac{(-1)^{n-3}}{(n-3)!}z, 0,  u''),
\end{array}
\]
where
\[
\begin{array}{l}
z\!=\!                   c\otimes {\bf c}(\diag(I_{n-3}, b),\diag(aI_{n-3}, 1), D_{1, n-2}, \dots ,D_{n-3, n-2}) \\
\hspace{0.3cm}+b\otimes {\bf c}(\diag(I_{n-3}, c),\diag(aI_{n-3}, 1), D_{1, n-2}, \dots ,D_{n-3, n-2}) \\
\hspace{0.3cm}+c\otimes {\bf c}(\diag(I_{n-3}, a),\diag(bI_{n-3}, 1), D_{1, n-2}, \dots ,D_{n-3, n-2}) \\
\hspace{0.3cm}+a\otimes {\bf c}(\diag(I_{n-3}, c),\diag(bI_{n-3}, 1), D_{1, n-2}, \dots ,D_{n-3, n-2}) \\
\hspace{0.3cm}+b\otimes {\bf c}(\diag(I_{n-3}, a),\diag(cI_{n-3}, 1), D_{1, n-2}, \dots ,D_{n-3, n-2}) \\
\hspace{0.3cm}+a\otimes {\bf c}(\diag(I_{n-3}, b),\diag(cI_{n-3}, 1), D_{1, n-2}, \dots ,D_{n-3, n-2}).
\end{array}
\]
(Note that $D_{i, n-2}=\diag(D_{i, n-3}, 1)=\diag(D_{i, n-3}(d_i), 1)$.)
Now through the composite $W_2'''\arr U_4'\oplus U_4''\arr T_2'\oplus T_2''$ we have
\[
\begin{array}{l}
a\otimes b\otimes c \otimes \{d_1, \dots, d_{n-3}\} \mt (u', u'') \mt (-\frac{(-1)^{n-3}}{(n-3)!}z+u_1'''', t''')=0,
\end{array}
\]
where
\[
\begin{array}{l}
u_1'''=\!\frac{(-1)^{n-2}}{(n-2)!}\Big(c\otimes {\bf c}(bI_{n-2},A_{1,n-2}, D_{2, n-2}, \dots ,D_{n-2, n-2}) \\
 \hspace{2cm}+b\otimes {\bf c}(cI_{n-2},A_{1,n-2}, D_{2, n-2}, \dots ,D_{n-2, n-2}) \\
 \hspace{2cm}+c\otimes {\bf c}(aI_{n-2},B_{1,n-2}, D_{2, n-2}, \dots ,D_{n-2, n-2}) \\
 \hspace{2cm}+a\otimes {\bf c}(cI_{n-2},B_{1,n-2}, D_{2, n-2}, \dots ,D_{n-2, n-2}) \\
 \hspace{2cm}+ b\otimes {\bf c}(aI_{n-2},C_{1,n-2}, D_{2, n-2}, \dots ,D_{n-2, n-2})\\
 \hspace{2cm}+a\otimes {\bf c}(bI_{n-2},C_{1,n-2}, D_{2, n-2}, \dots ,D_{n-2, n-2})\! \Big).
\end{array}
\]
Thus $u_1'''=\frac{(-1)^{n-3}}{(n-3)!}z$ and we have
$\chi_n'(u'')=(-u_{1}''', u'')=(\frac{(-1)^{n-3}}{(n-3)!}z, u'')$ which is trivial in $E_{2,n}^2(n, \kk)$.
This induces a well-defined  surjective map $\B_n(F)_\kk \arr E_{2,n}^2(n, \kk)$, which we denote again by
$\chi_n'$. This completes the proof of the theorem.
\end{proof}

\begin{cor}\label{prp:surj}
For any positive integers $n$ and any field $\kk$ such that $(n-2)!\in \kk^\times$, there is a natural map
\[
\chi_n\colon \B_n(F)_\kk \arr H_{n+1}( \GL_{n}(F),\kk)/ H_{n+1}(\fff\times \GL_{n-1}(F),\kk).
\]
If Conjecture  $\ref{conj1}$ holds with $\kk$ coefficients, then $\chi_n$ is surjective. In particular, there is a surjective map
\[
\begin{array}{c}
\fff \otimes H_\GL^{n-1}(F, n)_\kk\oplus \bigwedge_\z^2\fff\otimes  K_{n-1}^M(F)_\kk \oplus \B_n(F)_\kk \two H_\GL^n(F, n+1)_\kk.
\end{array}
\]
\end{cor}
\begin{proof}
These follow from Theorems \ref{G-K-R} and \ref{thm:surj0} by replacing $\zn $ with the field $\kk$.
\end{proof}

Based on the above results we make the following conjecture. 

\begin{conj}\label{conj:chi}
For any positive integer $n$, there is  a natural surjective map
\[
\begin{array}{c}
\chi_{n}\colon  \B_{n}(F)\nn \larr \displaystyle\frac{H_{n+1}(\GL_{n}(F),\zn)}{H_{n+1}(\fff\times \GL_{n-1}(F), \zn)}\cdot
\end{array}
\]
\end{conj}

\begin{cor} \label{cor:B3-B4}
Conjecture $\ref{conj:chi}$ holds for $1\leq n\leq 4$. In particular, there are natural surjective maps
\[
\begin{array}{l}
\chi_3\colon \B_3(F) \arr H_4( \GL_3(F),\z)/ H_{4}(\fff\times \GL_2(F),\z),
\end{array}
\]
\[
\begin{array}{l}
\ \ \ \ \ \ \ \ \ \ \ \ \chi_4\colon \B_4(F)\half \arr H_5( \GL_4(F),\z\half)/ H_5(\fff\times \GL_3(F),\z\half).
\end{array}
\]
\end{cor}
\begin{proof}
We know that Conjecture \ref{conj1} holds for $n=3$ \cite[Section~3]{mirzaii-2008} and  $n=4$ 
\cite[Section~6]{mirzaii2008}. Therefore the desired results follow from Theorem~\ref{thm:surj0}.
\end{proof}

\begin{cor} \label{thm:surj34}
Let $\kk$ be a field. 
\par {\rm (i)} There is a surjective map
\[
\begin{array}{c}
\fff \otimes H_\GL^2(F, 3)_\kk\oplus \bigwedge_\z^2\fff\otimes  K_2^M(F)_\kk \oplus \B_3(F)_\kk \two H_\GL^3(F, 4)_\kk.
\end{array}
\]
\par {\rm (ii)} If $\char(\kk)\neq 2$, then  there is a surjective map
\[
\begin{array}{c}
\fff \otimes H_\GL^3(F, 4)_\kk\oplus \bigwedge_\z^2\fff\otimes  K_3^M(F)_\kk \oplus \B_4(F)_\kk \two H_\GL^4(F, 5)_\kk.
\end{array}
\]
\end{cor}
\begin{proof}
These follow from Corollary \ref{prp:surj} and Corollary \ref{cor:B3-B4} (by replacing the coefficients $\zn $ with $\kk$).
\end{proof}

\begin{rem}
The groups $\B_n(F)$ should be seen as a generalisation of  the Bloch group $B(F)$ to higher dimensions, 
which is suitable for the study of the homology of $\GL_n(F)$. There are other versions of $\B_n(F)$ 
studied by Goncharov \cite[p. 222]{goncharov1995}, \cite[\S 3]{cathelineau2007} and Yagunov 
\cite[\S~2]{yagunov2000}. Our version seems to be different. At the moment 
we do not know if there is any connection between our version and their version of $\B_n(F)$ for $n\geq 3$.
\end{rem}

\begin{rem}\label{rem:hh}
Let $n\geq 3$ and  $0\leq k\!\leq n-1$. 
We expect the $k$-th homology of Complex (\ref{seq1}) to be related to the groups 
\[
\ker(H_{n+k-1}(\GL_{n-1}(F),\z)\arr H_{n+k-1}(\GL_{n}(F),\z)),
\] 
\[
H_{n+k}(\GL_n(F),\z)/H_{n+k}(\fff\times \GL_{n-1}(F),\z).
\]
The study of this connection for $k=1$, is the main topic of the current paper.  If $k=0$, then  all the groups  
$H_{n}(\GL_n(F),\z)/H_{n}(\fff \times \GL_{n-1}(F),\z)$, $\ker(H_{n-1}(\GL_{n-1}(F),\z)\arr H_{n-1}(\GL_{n}(F),\z))$ and 
$\ker(\delta_1^{(n)})/\im(\delta_2^{(n)})$ are trivial.  The triviality of the first two groups follow from Theorem \ref{suslin}.
To prove that $\ker(\delta_1^{(n)})/\im(\delta_2^{(n)})$ is trivial,  first note that the map 
\begin{align*}
K_n^M(F) &\arr \fff \otimes K_{n-1}^M(F)/\im(\delta_{2}^{(n)}), \\ 
\{a_1, a_2, \dots,a_n\}& \mt  a_1\otimes\{ a_2, \dots,a_n\}\!\!\!\! \pmod {\delta_2^{(n)}}
\end{align*}
is well-defined. To prove this, it is sufficient to show that the trivial element $\{1-a, a, a_3, \dots, a_n\}$ maps to zero. We have
\begin{align*}
0=\{1-a, a, a_3, \dots, a_n\}& \mt (1-a)\otimes\{ a, a_3, \dots, a_n\}\!\!\!\! \pmod {\delta_2^{(n)}}\\
& =-(1-a)\otimes\{ a_3, a, \dots, a_n\}\!\!\!\! \pmod {\delta_2^{(n)}}\\
& =a_3\otimes\{ 1-a, a, \dots, a_n\}\!\!\!\! \pmod {\delta_2^{(n)}}=0.
\end{align*}
Now it can be checked directly that  the map
\[
\overline{\delta_1^{(n)}}\colon \fff \otimes K_{n-1}^M(F)/\im(\delta_{2}^{(n)})\arr K_n^M(F).
\]
is the inverse of the above map.

We do not know how to connect these groups for $2\leq k\leq n-1$.
\end{rem}

\section{Homology of  \texorpdfstring{$\GL_{n}(F)$}{Lg} over certain fields}\label{exa-Bn}

In this section we study the homology of $\GL_n$ over algebraically closed, real closed, global and local fields.
As we will see many of the above results can be improved over these fields.

\subsection{Fields with divisible multiplicative groups}

Let $F$ be algebraically closed. Since $K_m^M(F)$ is uniquely divisible for $m\geq 2$ \cite[Corollary 1.3]{bass-tate1973},
it follows that $\B_n(F)$ is uniquely divisible for $n\geq 3$. Moreover, since the Bloch group $B(F)$ is uniquely divisible 
\cite[Corollary~5.7]{suslin1991} and $\B_2(F)=B(F)$, $\B_2(F)$ is uniquely divisible too. Therefore $\B_n(F)$ is uniquely 
divisible for any $n$. Using this fact, we can improve almost all the results obtained in the previous sections.
For example Theorem \ref{thm:phi}, Corollary \ref{cor:n=4}, Theorem \ref{thm:surj0}
and Corollary \ref{cor:B3-B4} hold with integral coefficients.

But using the next theorem and Theorem \ref{G-K-R} of Galatius, Kupers and Randal-Williams, we can prove far better results.

\begin{thm}[Galatius, Kupers, Randal-Williams]\label{g-k-r}
Let $p$ be a prime such that $\fff\otimes \z/p=0$. Then $H_d(\GL_n(F), \GL_{n-1}(F),\z/p)=0$ in degrees $d<3n/2$.
\end{thm}
\begin{proof}
See \cite[Theorem 9.11]{G-K-R2020}.
\end{proof}

\begin{prp}\label{prp:ac}
Let $p$ be a prime such that $\fff\otimes \z/p=0$. Then the relative homology group $H_{n+k}(\GL_{n}(F),\GL_{n-1}(F),\z)$ and
the kernel of the natural map 
$H_{n+k-1}(\GL_{n-1}(F),\z) \arr H_{n+k-1}(\GL_{n}(F),\z)$ are uniquely $p$-divisible for $n> 2(k+1)$.
\end{prp}
\begin{proof}
By  Theorem \ref{g-k-r} we have $H_d(\GL_n(F), \GL_{n-1}(F),\z/p)=0$ in degrees $d<3n/2$. This in particular implies that 
\[
H_{n+k}(\GL_n(F), \GL_{n-1}(F),\z/p)=0
\]
for  $n>2k$ and 
\[
H_{n+k+1}(\GL_n(F), \GL_{n-1}(F),\z/p)=0
\]
for  $n> 2(k+1)$.
From the long exact sequence induced by the exact sequence of coefficients 
\[
0\arr \z \overset{p}{\arr} \z \arr \z/p\arr 0
\]
it follows that multiplication by $p$ on $H_{n+k}(\GL_n(F), \GL_{n-1}(F),\z)$ is a bijection for  $n> 2(k+1)$. Thus 
$H_{n+k}(\GL_n(F), \GL_{n-1}(F),\z)$ is uniquely $p$-divisible. This implies 
$\ker(H_{n+k-1}(\GL_{n-1}(F),\z) \arr H_{n+k-1}(\GL_{n}(F),\z))$ has the same property.
\end{proof}

The following is Theorem C of the introduction.

\begin{thm}\label{thm:div}
Let $F$ be a field such that $\fff$ is divisible. Then 
\par {\rm (i)} $H_{n}(\GL_{n-1}(F),\z) \arr H_{n}(\GL_{n}(F),\z)$ is injective for any $n\geq 1$,
\par {\rm (ii)} $H_\GL^n(F, n+1)$ is divisible for $n\neq 2$ and is uniquely divisible for $n\geq 5$.
\end{thm}
\begin{proof}
Since $\fff$ is divisible,  $\fff^{\otimes k}$ is uniquely divisible for $k\geq 2$ (see the proof of 
\cite[Proposition 1.2]{bass-tate1973}). Thus $K_m^M(F) $ is divisible for $m\geq 1$ and
$\fff^{\otimes k}\otimes K_m^M(F)$ is uniquely divisible
for $k, m\geq 1$. This implies that $\B_m(F)$ is uniquely divisible for any $m\geq 3$.
\par (i) The claim is trivial for $n=1, 2$. By Corollary~\ref{cor:n=4}, $\kappa_3$  is surjective 
with $2$-torsion image. Since $\B_3(F)$ is uniquely divisible, $\kappa_3$ must be the trivial map. Thus we 
get the injectivity for $n=3$. Now let $n=4$. By Theorem~\ref{thm:phi}, there is a natural map
\[
\varphi_4\colon \B_4(F)\arr \ker(H_{4}(\GL_{3}(F),\z)\arr H_{4}(\GL_{4}(F),\z)),
\]
with $3$-torsion image.
Since $\fff$ is divisible, for any $a\in \fff$, the polynomial $X^2-a\in F[X]$ splits into linear factors. Thus by 
\cite[Proposition 1.2]{bass-tate1973}, $K_m^M(F)$ is uniquely $2$-divisible for any $m\geq 2$. Thus we have the decomposition
\[
H_{3}(\GL_{3}(F),\z)\simeq H_{3}(\GL_{2}(F),\z)\oplus K_3^M(F),
\] 
where the splitting map $K_3^M(F)\arr H_{3}(\GL_{3}(F),\z)$ is given by $\{a,b,c\}\mt [a^{1/2},b,c]$. Now  as the proof of 
Theorem~\ref{thm:phi}(ii), using \cite[Theorem 2]{mirzaii2008}, we can show that $\varphi_4$ is surjective. Since $\B_4(F)$ is uniquely 
divisible, $\varphi_4$ is the trivial map. This implies the injectivity for $n=4$.

So let $n\geq 5$. Let $p$ be a prime. Since $\fff$ is divisible, $\fff\otimes \z/p=0$.
By Proposition \ref{prp:ac} (for k=1), the kernel of 
\[
H_n(\GL_{n-1}(F),\z) \arr  H_n(\GL_n(F),\z)
\]
is uniquely $p$-divisible. This holds for any prime and thus this kernel is uniquely divisible. 
On the other hand  by Theorem \ref{G-K-R} this kernel is torsion. Therefore it must be trivial.

\par (ii) If $n\geq 5$, then by Proposition \ref{prp:ac}, $H_{n+1}(\GL_n(F),\GL_{n-1}(F),\z)$ is 
uniquely $p$-divisible for any prime $p$. This implies that it is uniquely divisible. By (i), 
\[
\begin{array}{c}
H_\GL^n(F, n+1)\simeq H_{n+1}(\GL_n(F),\GL_{n-1}(F),\z).
\end{array}
\]
Therefore $H_\GL^n(F, n+1)$ is uniquely divisible.

If $n=1$, then $H_\GL^1(F, 2)\simeq \bigwedge_\z^2\fff$, so it is uniquely divisible. For $n=3$, consider the exact sequence 
\begin{equation*}
\fff \otimes H_\GL^{2}(F, 3)
\oplus \bigwedge{}_\z^2\fff\otimes K_{2}^M(F)\oplus \tors(\fff, K_{2}^M(F))
\end{equation*}
\[
\arr H_\GL^{3}(F, 4) \arr \frac{H_{4}(\GL_3(F),\z)}{H_{4}(\fff\times \GL_{2}(F),\z)}\arr 0.
\]
By Corollary \ref{cor:B3-B4}, the map
\[
\chi_3\colon \B_3(F) \arr \frac{H_{4}(\GL_3(F),\z)}{H_{4}(\fff\times \GL_{2}(F),\z)}
\]
is surjective. Since $\B_3(F)$ is uniquely divisible, $H_{4}(\GL_3(F),\z)/H_{4}(\fff\times \GL_{2}(F),\z)$ is divisible.
Now it follows from the above exact sequence that $H_\GL^{3}(F, 4) $ is divisible. (Note that $\fff \otimes H_\GL^{2}(F, 3)$,
$\bigwedge{}_\z^2\fff\otimes K_{2}^M(F)$ and  $\tors(\fff, K_{2}^M(F))$ are divisible.)

In a similar way we can show that $H_\GL^{4}(F, 5) $ is divisible. Here we need to show that we have a surjective map 
$\B_4(F) \arr \displaystyle\frac{H_{5}(\GL_4(F),\z)}{H_{5}(\fff\times \GL_{3}(F),\z)}$. But this can be proved as the proof of
Theorem~\ref{thm:surj0}, using the validity of Conjecture \ref{conj1} for $n=4$ (\cite[Section 6]{mirzaii2008}) and 
the decomposition $H_{3}(\GL_{3}(F),\z)\simeq H_{3}(\GL_{2}(F),\z)\oplus K_3^M(F)$.
This completes the proof of the theorem.
\end{proof}

\begin{rem}
The multiplicative group of an algebraically closed field is divisible. But the 
class of fields such that their multiplicative groups are divisible is much larger. For more on this see 
\cite[Theorems 1,  2]{oman2016}.
\end{rem}

\begin{cor}\label{thm:ac}
Let $F$ be algebraically closed (or more generally  a field such that $\fff\otimes \z/p=0$ and 
the polynomial $X^p-1\in F[X]$ splits into linear factors, for any prime $p$). Then for any $n$,
\par {\rm (i)} $H_{n}(\GL_{n-1}(F),\z) \arr H_{n}(\GL_{n}(F),\z)$ is injective,
\par {\rm (ii)} $H_{n}(\GL_{n}(F),\z)\simeq H_{n}(\GL_{n-1}(F),\z)\oplus K_n^M(F)$,
where the splitting map $K_n^M(F)\arr H_{n}(\GL_{n}(F),\z)$ is given by
\[
\{a_1,\dots,a_n\} \mt[a_1^{\frac{(-1)^{n-1}}{(n-1)!}},\dots,a_n]=[a_1,a_2^{-\frac{1}{1}},a_3^{-\frac{1}{2}}, \dots,a_n^{-\frac{1}{n-1}}].
\]
\par {\rm (iii)} $H_\GL^n(F, n+1)$ is divisible for $n\geq 1$ and is uniquely divisible for $n\neq 2$.
\par {\rm (iv)} There is a natural map
\[
\chi_n\colon \B_n(F) \arr H_{n+1}( \GL_{n}(F),\z)/ H_{n+1}(\fff\times \GL_{n-1}(F),\z).
\]
Moreover, if Conjecture  $\ref{conj1}$ holds with $\z$-coefficients over $F$, then $\chi_n$ is surjective.
\end{cor}
\begin{proof}
(i) This follows from Theorem \ref{thm:div}(i).

\par  (ii) This follows from (i) and the fact that $K_n^M(F)$ is uniquely divisible for $n\geq 2$ 
\cite[Proposition 1.2, Corollary 1.3]{bass-tate1973}.
\par (iii) The claim is trivial for $n=1$. Let $n=2$.  Since $K_i^M(F)$ is uniquely divisible for $i\geq 2$, 
$\fff\otimes K_2^M(F)$ is uniquely divisible and $K_3^M(F)/2=0$. Thus by Proposition \ref{mirzaii}, we have the exact sequence 
\[
H_3(\SL_2(F),\z)\arr H_\GL^2(F,3)\arr \fff\otimes K_2^M(F)\arr 0.
\]
By \cite[Theorem 6.1(ii)]{mirzaii-2008}, the natural map 
\[
H_3(\SL_2(F),\z)=H_3(\SL_2(F),\z)_\fff\arr H_3(\SL(F),\z)
\]
is injective. Now it follows from the commutative diagram
\[
\begin{tikzcd}
H_3(\SL_2(F),\z) \ar[r]\ar[d]&  H_3(\SL_3(F),\z)\ar[d] \ar[r]& H_3(\SL(F),\z)\ar[d, hook] \\
H_3(\GL_2(F),\z)\ar[r, hook]& H_3(\GL_3(F),\z) \ar[r, "\simeq"] & H_3(\GL(F),\z)
\end{tikzcd}
\]
that the map $H_3(\SL_2(F),\z)\arr H_3(\GL_2(F),\z)$ is injective. 
This implies that the map $H_3(\SL_2(F),\z) \arr H_\GL^2(F,3)$ is injective.
Since 
\[
H_3(\SL_2(F),\z)\simeq K_3^\ind(F)
\]
\cite[Proposition 6.4]{mirzaii-2008}, similar to  the proof of Proposition \ref{mirzaii}, we can show that 
\[
H_\GL^2(F,3)\simeq K_3^\ind (F)\oplus \fff\otimes K_2^M(F).
\]
But $K_3^\ind (F)$ is divisible (see the proof of \cite[Corollary 9.13]{G-K-R2020}). Therefore $H_\GL^2(F,3)$ is divisible. 
Now let $n= 3$. Then $\fff \otimes H_\GL^{2}(F, 3)$, $\bigwedge{}_\z^2\fff\otimes K_{2}^M(F)$ are uniquely divisible 
(see the proof of \cite[Proposition 1.2]{bass-tate1973})  and  $\tors(\fff, K_{2}^M(F))=0$.
Now as the proof of Theorem \ref{thm:div}(ii) we can show that $H_\GL^{3}(F, 4)$ is uniquely divisible. In a similar
way one can show that  $H_\GL^{4}(F, 5)$ is uniquely divisible. This resolves the case $n=4$.
If $n\geq 5$, the claim follows from Theorem \ref{thm:div}(ii).
\par (iv) We may assume $n\geq 3$.  The proof is similar to the proof of Theorem~\ref{thm:surj0}. Here we should use
(ii). In fact $\chi_n=d_{2,n}^2(n,\z)\circ\chi_n'$, where $\chi_n'\colon \B_n(F) \two E_{2,n}^2(n,\z)$ is surjective.
\end{proof}

\begin{rem}
(i) Let $F$ be a field such that $\fff$ is divisible. Since $\fff=(\fff^2)$ and
$\fff\otimes K_2^M(F)$ is uniquely divisible (see the proof of \cite[Proposition~1.2]{bass-tate1973}),
as in the proof of Proposition \ref{mirzaii}, we can show that
\[
H_\GL^2(F,3)\simeq K_3^\ind (F)\oplus \fff\otimes K_2^M(F).
\]
If $\mu(F)\neq 1$ (e.g. $\char(F)\neq 2$), Suslin's Bloch-Wigner exact sequence
\[
0\arr \tors(\mu(F),\mu(F))^\sim \arr K_3^\ind(F)\arr B(F)\arr 0
\]
(see \cite[Theorem 5.2]{suslin1991}) shows that $K_3^\ind(F)$ has nontrivial torsion elements.
Therefore $H_\GL^2(F,3)$ is not uniquely divisible. 

We do not know if $H_\GL^2(F,3)$ is divisible, when $\fff$ is divisible.
By  Corollary \ref{thm:ac} this is the case if $F$ is algebraically closed.
\par (ii) The part (iii) of Corollary \ref{thm:ac} gives a positive answer to a question asked by the author in \cite[page 616]{mirzaii2007}.
\end{rem}

\subsection{Real closed fields}

Let $F$ be a real closed field \cite[Chapter XI]{lang2002}. It is well-known that $F$ has a unique order and $\fff\simeq \{\pm 1\}\times F^{>0}$. 
Any polynomial of odd degree has a root. Moreover, for any $a\in F^{>0}$, $X^2-a$ has a root in $F^{>0}$, i.e. $\sqrt{a}\in F^{>0}$. Thus 
$F^{>0}$ is a uniquely divisible group.  The following result is well-known.

\begin{lem}\label{milnor}
Let $F$ be a real closed field. Then for any $m\geq 1$, $K_m^M(F)$ is direct sum of a cyclic group of order  $2$ generated by $\{-1, \dots, -1\}$ 
and a divisible subgroup $K_m^M(F)^\circ$ generated by all symbols $\{a_1,\dots,a_m\}$, $a_1,\dots,a_m\in F^{>0}$. 
\end{lem}
\begin{proof}
For proof of the case $F=\R$ see \cite[Theorem 14.46, Corollary 14.47]{magurn2002} or \cite[Example 1.6, Theorem 1.4]{milnor1970}.
The proof of the general case is similar.
\end{proof}

It follows from the above lemma that ${F^\times}^{\otimes i} \otimes K_m^M(F)$ decomposes as a direct sum of the subgroup of order two, generated by 
$(-1)\otimes\cdots\otimes  (-1)\otimes \{-1, \dots, -1\}$ and the uniquely divisible subgroup ${(F^{>0})}^{\otimes i} \otimes K_m^M(F)^\circ$ 
for $i, m\geq 1$ (see the proof of \cite[Proposition 1.2]{bass-tate1973}). Since
\[
\delta_3^{(n)}\Big((-1)\otimes (-1)\otimes (-1)\otimes \{-1, \dots,-1\}\Big)\!=\!(-1)\otimes (-1)\otimes \{-1, \dots,-1 \},
\]
$\B_n(F)$ is uniquely divisible for any $n\geq 3$.  

Observe that $\B_2(\R)$ is divisible. This follows from the fact that $K_3(\R)$ is divisible \cite[Theorem 4.9]{suslin1984}, since $\B_2(\R)$ 
is a quotient of  $K_3^\ind(\R)$.

The following is Theorem D of the introduction.

\begin{thm}\label{real}
Let  $F$ be a real closed field. Then
\par {\rm (i)}  $H_3(\GL_2(F),\z) \arr  H_3(\GL_3(F),\z)$ is injective,
\par {\rm (ii)}  $H_n(\GL_{n-1}(F),\z\half) \arr  H_n(\GL_n(F),\z\half)$ is injective for any $n$,
\par {\rm (iii)} $H_\GL^n(F, n+1)\half$  is uniquely divisible for $n\geq 5$.

\end{thm}
\begin{proof}
(i) By Corollary \ref{cor:n=4}, $\kappa_3$ is surjective with $2$-torsion image. Since $\B_3(F)$ is
divisible, $\kappa_3$ must be the trivial map. Thus the desired map is injective.
\par (ii) The claim is trivial for $n=1,2$. The case $n=3$ follows from (i). The case $n=4$ can be proved as the part (i).
So we may assume that $n\geq 5$. Let $p$ be an odd prime. Then $\fff\otimes \z/p=0$.
By Proposition \ref{prp:ac}, the kernel of $H_n(\GL_{n-1}(F),\z) \arr  H_n(\GL_n(F),\z)$ is uniquely $p$-divisible.
Since this holds for any odd prime, the kernel of the above map is $2$-torsion. From this we obtain the desired
injectivity.
\par (iii) The proof is similar to the proof of Theorem \ref{thm:div}(ii).
\end{proof}

\begin{cor}
Let $n\geq 1$. Then
\par {\rm (i)}  $H_n(\GL_{n-1}(\R),\z\half) \arr  H_n(\GL_n(\R),\z\half)$ is injective,
\par {\rm (ii)}   $H_n(\GL_n(\R),\z\half)\simeq H_n(\GL_{n-1}(\R),\z\half)\oplus K_n^M(\R)^\circ$,
\par {\rm (iii)} $H_\GL^n(\R, n+1)\half$ is divisible and is uniquely divisible for any $n\neq 2$,
\par {\rm (iv)} there is a natural  map 
\[
\begin{array}{c}
\chi_n\colon B_n(\R) \arr H_{n+1}(\GL_n(\R),\z\half)/H_{n+1}(\R^\times \times \GL_{n-1}(\R),\z\half).
\end{array}
\]
Moreover, if Conjecture  $\ref{conj1}$ holds with $\z\half$ coefficients for $\R$, then $\chi_n$ is surjective.
\end{cor}
\begin{proof}
(i) This is an special case of Theorem \ref{real}(ii).
\par (ii) This follows from (i), Theorem \ref{suslin}, the fact that $K_n^M(\R)^\circ$ is uniquely divisible  
\cite[Example 7.2(c), Chap. III]{weibel2013} and the isomorphism $K_n^M(\R)^\circ\simeq K_n^M(\R)\half$.
\par (iii) Since $H_\GL^1(\R, 2)\simeq \bigwedge_\z^2\R^{>0}$, the claim is trivial for $n=1$. By \cite[Theorem 4.9]{suslin1984},
$K_3(\R)$ is divisible. Thus by Lemma \ref{milnor} and Proposition \ref{mirzaii}, $H_\GL^2(\R, 3)\half$ is divisible.
The cases $n=3, 4$ can be proved similar to the cases $n=3, 4$ done in the proof of Corollary \ref{thm:ac}(iii). 
If $n\geq 5$, the claim follows from Theorem \ref{real}(iii).
\par (iiii) We may assume $n\geq 3$. The proof is similar to the proof of Theorem~\ref{thm:surj0}. In fact 
$\chi_n=d_{2,n}^2(n,\z\half)\circ\chi_n'$, where $\chi_n'\colon \B_n(\R) \two E_{2,n}^2(n,\z\half)$ is surjective.
\end{proof}

\subsection{Global and local fields}

The homology of general linear groups over global fields is well studied. For example in \cite[Corollary 7.6]{B-Y1994} Borel and 
Yang has shown that for a number field $F$, the natural map 
\[
H_{d}(\GL_{n-1}(F),\q) \arr H_{d}(\GL_{n}(F),\q)
\]
is injective for any $d$ and is surjective if $d\leq 2n-3$. Recently, Galatius, Kupers and Randal-Williams have proved the following 
result.

\begin{thm}[Galatius, Kupers and Randal-Williams]
Let $F$ be a field  with torsion  $K_2(F)$. Then $H_{d}(\GL_{n-1}(F),\q) \arr H_{d}(\GL_{n}(F),\q)$ is surjective
if $d< (4n-1)/3$ and is injective if $d<(4n-4)/3$.
\end{thm}
\begin{proof}
See \cite[Theorem E or Theorem 9.10]{G-K-R2020}.
\end{proof}

It is known that the second $K$-group of a global field is torsion \cite[p.~144, p.~158]{weibel2005}. Thus the above 
theorem can be applied to global fields. 

Let $F$ be a global field. Then by a theorem of Bass and Tate, for any $m\geq 3$,
\[
K_m^M(F)\simeq (\z/2)^{r_1},
\]
 where $r_1$ is the number of  embeddings of $F$ in $\R$  \cite[Theorem~2.1, Chap. II]{bass-tate1973}. Thus 
$\B_4(F)$ is torsion and for any $n\geq 5$, $\B_n(F)$ is 2-torsion. 

If $F$ is a local field, then $K_m^M(F)$ is uniquely divisible for any $m\geq 3$ \cite[Proposition 7.1, Chap. VI]{weibel2013}. 
In particular for any $n\geq 5$,  $\B_n(F)$ is uniquely divisible.

\begin{prp}\label{prp:a}
Let $F$ be either a global field or a local field and let the sequence
\[
\begin{array}{c}
H_n(\fff^2\times \GL_{n-2}(F),\z\sixth)\overset{{\alpha_1}_\ast-{\alpha_2}_\ast}{-\!\!\!-\!\!\!-\!\!\!\larr} H_n(\fff\times\GL_{n-1}(F),\z\sixth) \\
\overset{\inc_\ast}{\larr}H_n(\GL_{n}(F),\z\sixth) \arr  0
\end{array}
\]
be exact for any $3\leq n\leq s$. Then for any $1\leq n\leq s$,
\par {\rm (i)} $H_{n}(\GL_{n-1}(F),\z\sixth) \arr H_{n}(\GL_{n}(F),\z\sixth)$ is injective,
\par {\rm (ii)} $H_{n}(\GL_{n}(F),\z\sixth)\simeq H_{n}(\GL_{n-1}(F),\z\sixth)\oplus K_n^M(F)\sixth$. In case of local fields
the splitting map $K_n^M(F)\sixth\arr H_{n}(\GL_{n}(F),\z\sixth)$ is given by
\[
{(-1)^{n-1}}{(n-1)!}\{a_1,\dots,a_n\} \mt[a_1,\dots,a_n],
\]

\par {\rm (iii)} There is a natural map
\[
\begin{array}{c}
\chi_n\colon \B_n(F)\sixth \arr H_{n+1}( \GL_{n}(F),\z\sixth)/ H_{n+1}(\fff\times \GL_{n-1}(F),\z\sixth),
\end{array}
\]
which is surjective if Conjecture $\ref{conj1}$ holds with $\z\sixth$ coefficients.
\end{prp}
\begin{proof}
The claim is trivial for $n=1,2$. The cases $n=3, 4$ follow from Corollaries \ref{cor:n=4} and  \ref{cor:B3-B4}. So let $n\geq 5$.
The proof is by induction and is similar to the proofs of Theorems~\ref{thm:phi} and \ref{thm:surj0}.
\end{proof}

\begin{rem}
In the previous proposition, we believe it is enough to invert $2$ (rather than $6$). For this we need to prove the injectivity of
\[
\begin{array}{c}
H_4(\GL_3(F), \z\half)\arr H_4(\GL_3(F), \z\half)
\end{array}
\]
(see Remark \ref{rem:k3}(ii)). So far we could prove the injectivity of the map $H_4(\GL_3(F), \z\sixth)\arr H_4(\GL_3(F), \z\sixth)$
(Corollary \ref{cor:n=4}), which explains why in the previous proposition we invert $6$. In the following we prove a special 
case, where we invert only $2$ in the coefficients ring.
\end{rem}

\begin{prp}
Let $F$ be a local field such that $3\nmid |\mu(F)|$. Then the natural map $H_4(\GL_3(F), \z\half)\arr H_4(\GL_4(F), \z\half)$
is injective. Moreover, $H_4(\GL_4(F), \z\half)\simeq H_4(\GL_3(F), \z\half)\oplus K_4^M(F)$.
\end{prp}
\begin{proof}
Let $m=|\mu(F)|$. 
By a theorem of Moore, $K_2^M(F)$ is the direct sum of a uniquely divisible abelian group and a finite cyclic group, isomorphic to
$\mu(F)$ \cite[Theorem 6.2.4, Chap. III]{weibel2013}. This would imply that $\fff\otimes \fff\otimes K_2^M(F)$ is the direct sum of a 
uniquely divisible abelian group and the $m$-torsion group $\fff\otimes \fff\otimes \z/m\simeq \fff/(\fff^m)\otimes \fff/(\fff)^m$ (observe that
$\fff/(\fff^m)$ is finite). Since $K_3^M(F)$ is uniquely divisible, $\fff\otimes K_3^M(F)$ also is uniquely divisible. Now from the definition of
$\B_4(F)$ we have
\[
\B_4(F)\simeq U \oplus A, 
\]
where $U$ is divisible and $A$ is $m$-torsion. Since $3\nmid m$ and since the image of the surjective map 
\[
\begin{array}{c}
\kappa_4:\B_4(F)\half \arr \ker(H_4(\GL_3(F), \z\half)\arr H_4(\GL_3(F), \z\half))
\end{array}
\]
(Corollary \ref{cor:n=4}) of Theorem \ref{thm:phi} is $3$-torsion, we see that $\kappa_4$ is the trivial map. This proves the injectivity result.
The other claim follows from Corollary~\ref{cor:dec} and the fact that $K_4^M(F)$ is uniquely divisible.
\end{proof}


\bigskip
\address{{\footnotesize

Instituto de Ci\^encias Matem\'aticas e de Computa\c{c}\~ao (ICMC),

Universidade de S\~ao Paulo (USP), S\~ao Carlos, S\~ao Paulo, Brasil

E-mail:\ bmirzaii@icmc.usp.br,
}}

\end{document}